\documentclass[aos,preprint]{imsart}
\setattribute{journal}{name}{}
\RequirePackage[round]{natbib}
\usepackage{comment}
\usepackage{lineno}
\usepackage{amsmath, amsthm, amssymb, mathtools, bm, bbm, mathrsfs}
\usepackage{graphicx} 
\usepackage{tikz, tkz-euclide, tikz-cd, tikz-dimline}  \usetkzobj{all}  \usetikzlibrary{arrows, decorations.pathreplacing, positioning, shapes}
\usepackage[capitalise]{cleveref}
\usepackage{caption}
\usepackage{booktabs, multirow}
\usepackage{enumerate}
\usepackage{algpseudocode, algorithm}

\newtheorem{theorem}{Theorem}

\newtheorem{proposition}{Proposition}
\newtheorem{definition}{Definition}
\newtheorem{lemma}{Lemma}

\newtheorem{corollary}{Corollary}

\theoremstyle{definition}
\newtheorem{remark}{Remark}

\theoremstyle{definition}
\newtheorem{example}{Example}

\makeatletter
\newcommand{\longdash}[1][2em]{\makebox[#1]{$\m@th\smash-\mkern-7mu\cleaders\hbox{$\mkern-2mu\smash-\mkern-2mu$}\hfill\mkern-7mu\smash-$}}
\makeatother
\newcommand{\omitskip}{\kern-\arraycolsep}

\newcommand{\T}{\intercal}

\DeclareMathOperator*{\KL}{\mathcal{D}_\text{KL}}
\DeclareMathOperator*{\TV}{\mathrm{d}_\text{TV}}
\newcommand{\iid}{\overset{\textsf{iid}}{\sim}}
\newcommand{\indep}{\mathrel{\text{\scalebox{1.07}{$\perp\mkern-10mu\perp$}}}}

\def\model{{\mathcal{M}}}

\DeclareMathOperator{\diag}{\mathsf{diag}}

\DeclareMathOperator{\E}{\mathbb{E}}
\DeclareMathOperator{\G}{\mathsf{G}}

\DeclareMathOperator{\Tr}{\mathsf{Tr}}
\newcommand*\dd{\mathop{}\!\mathrm{d}}

\newcommand{\I}{\mathbb{I}}

\newcommand{\distconvto}{\overset{d}{\Rightarrow}}
\newcommand{\weaklyconvto}{\rightsquigarrow}

\newcommand{\N}{\mathcal{N}}

\newcommand{\PSD}{\mathbb{R}_{\text{PSD}}}
\newcommand{\PD}{\mathbb{R}_{\text{PD}}}

\newcommand{\err}{\text{err}}
\newcommand{\pow}{\text{pow}} 
\begin{document}

\begin{frontmatter}

\title{On Testing Marginal versus Conditional Independence}
\runtitle{Marginal versus Conditional Independence}

\begin{aug}
\author{\fnms{F. Richard}~\snm{Guo}\ead[label=e1]{ricguo@stat.washington.edu}}
\and
\author{\fnms{Thomas S.}~\snm{Richardson}\ead[label=e2]{thomasr@u.washington.edu}}
\runauthor{Guo and Richardson}
\affiliation{Department of Statistics\\University of Washington, Seattle}

\address{Department of Statistics\\
University of Washington\\
Box 354322\\
Seattle, WA 98195 \\
\printead{e1}\\
\phantom{E-mail: }\printead*{e2}\\}
\today
\end{aug}

\begin{abstract}
We consider testing marginal independence versus conditional independence in a trivariate Gaussian setting. The two models are non-nested and their intersection is a union of two marginal independences. 
We consider two sequences of such models, one from each type of independence, that are closest to each other in the Kullback-Leibler sense as they approach the intersection. They become indistinguishable if the signal strength, as measured by the product of two correlation parameters, decreases faster than the standard parametric rate. Under local alternatives at such rate, we show that the asymptotic distribution of the likelihood ratio depends on where and how the local alternatives approach the intersection. To deal with this non-uniformity, we study a class of ``envelope'' distributions by taking pointwise suprema over asymptotic cumulative distribution functions. We show that these envelope distributions are well-behaved and lead to model selection procedures with rate-free uniform error guarantees and near-optimal power. To control the error even when the two models are indistinguishable, rather than insist on a dichotomous choice, the proposed procedure will choose either or both models. 
\end{abstract}


\end{frontmatter}

\tableofcontents

\section{Introduction}
It is often of interest to test marginal or conditional independence for a set of random variables. For example, in the context of graphical modeling, the PC algorithm \citep{spirtes2000causation} for directed acyclic graph model selection determines the orientation of an unshielded triple $X - Z - Y$ based on whether the separating set of $X$ and $Y$ contains $Z$: if so, $X \indep Y \mid Z$ and $Z$ is not a collider; if not, $X \indep Y$ and the triple is oriented as $X \rightarrow Z \leftarrow Y$. The reader is referred to \citet{dawid1979conditional, lauritzen1996graphical, koller2009probabilistic} and \citet{Reichenbach1956} for more discussion.

Here we consider the simplest case, namely testing $X_1 \indep X_2$ versus $X_1 \indep X_2 \mid X_3$ in a trivariate Gaussian setting. For testing whether \emph{a specific} marginal or conditional independence holds, it is common to use the correlation coefficient or partial correlation coefficient under Fisher's $z$-transformation \citep{fisher1924distribution} as the test statistic. Under independence, the transformed correlation coefficient is approximately distributed as a normal distribution with zero mean and variance determined by the sample size and the number of variables being conditioned on \citep{hotelling1953new, anderson1984introduction}. In this paper, however, we assume \emph{at least one} type of independence holds (from prior knowledge or precursory inference) and we want to contrast the two types. To this end, we will use the likelihood ratio statistic, which often provides intuitively reasonable tests for composite hypotheses \citep{perlman1999emperor}, especially in terms of model selection. 

\subsection*{Contributions} We briefly highlight our main contributions as follows. Firstly, we consider an important problem in non-nested model selection, which is in general less well-understood than the nested case. Secondly, we take an approach that is different from the usual Neyman-Pearson framework, in the sense that we treat the two models symmetrically and allow them to be both selected if the data does not significantly prefer one over the other. Thirdly, by introducing a new family of envelope distributions, we deal with non-uniform asymptotic laws of the likelihood ratio statistic. The model selection procedures we propose come with asymptotic guarantees that are applicable to all varieties of relations between the sample size and the signal strength; an assumption on the asymptotic rate is not required. 

\subsection*{Notation} The following notation is used through the paper. $\PD^{n \times n}$ denotes $n \times n$ positive definite matrices. $\Theta$ denotes the parameter space and $\model$ denotes a model, which is subset of the parameter space. $\model_1 \setminus \model_2$ denotes the set of parameters that belong to $\model_1$ but not belong to $\model_2$. 

We use $P$ and $Q$ to denote measures, and similarly $P_n$ and $Q_n$ to denote sequences of measures. $\mu$ is reserved for the Lebesgue measure. Lower-case letters $p, q$ denote the densities of $P, Q$ with respect to $\mu$. $P_n^n$ denotes the $n$-sample product (tensorized) measure of $P_n$, namely the law of $X_1, \dots, X_n \iid P_n$. We write $P_n \distconvto P$ if $P_n$ converges (weakly) to $P$ in law. For $X_n(t)$ a stochastic process indexed by $t \in T$, we write $X_n \weaklyconvto X$ if $X_n(t)$ converges weakly to $X(t)$. 

For two sequences $\{a_n\}$ and $\{b_n\}$, we write $a_n = O(b_n)$ if there exists a constant $c < \infty$ such that $a_n \leq c b_n$ for large enough $n$; $a_n = o(b_n)$ and $b_n = \omega(a_n)$ if $\lim_{n \rightarrow \infty} a_n / b_n = 0$; $a_n \asymp b_n$ if $a_n = O(b_n)$ and $b_n = O(a_n)$.

Also, we write $ x \lesssim y$ if $x \leq c y$ for some constant $c > 0$. We write $x \vee y = \max(x,y)$ and $x \wedge y = \min(x,y)$.

\subsection*{Setup} For $(X_1, X_2, X_3) \sim \N\{0, \Sigma = (\sigma_{ij})\}$ with parameter space $\Theta$ being the set of $3 \times 3$ real positive definite matrices $\PD^{3 \times 3}$, we consider testing
\begin{equation}
\mathcal{M}_0: X_1 \indep X_2 \quad  \text{versus} \quad \mathcal{M}_1: X_1 \indep X_2 \mid X_3.
\end{equation}
$\model_0$ and $\model_1$ are algebraic models \citep{drton2007algebraic} as represented by equality constraints
\begin{equation}
\model_0: \{\sigma_{12}=0\}, \quad \model_1: \{\sigma_{12} \sigma_{33} = \sigma_{13} \sigma_{23}\}
\end{equation}
imposed on $\Theta$. They are visualized in the correlation space (ignoring the variances) in \Cref{fig:geometry}.

\begin{figure}[!ht]
\includegraphics[width=0.5\textwidth]{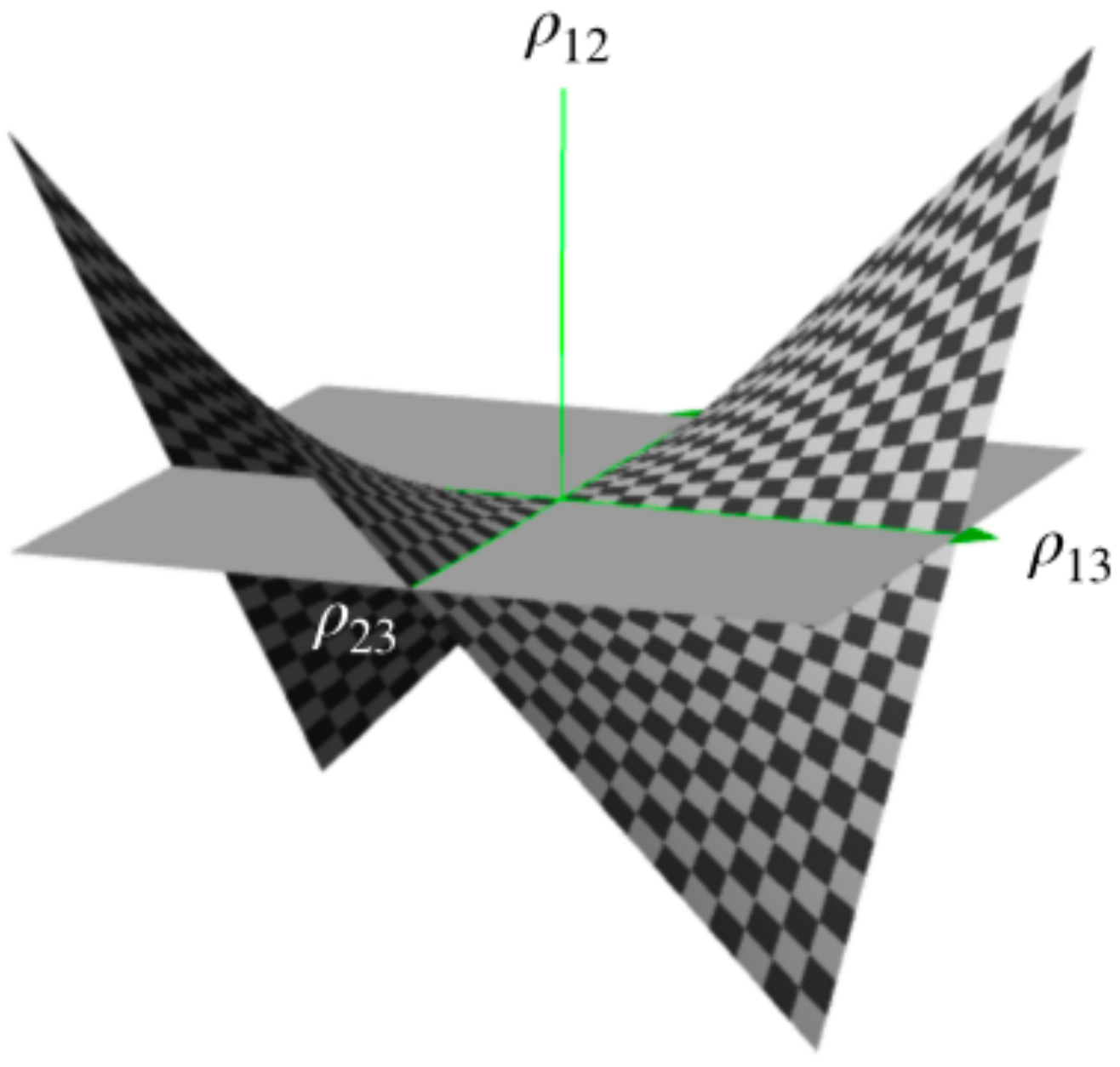}
\caption{The two models visualized in the correlation space: $\model_0: \rho_{12} = 0$ (grey plane) and $\model_1: \rho_{12} = \rho_{13} \rho_{23}$ (checkerboard). $\model_0 \cap \model_1$ consists of the $\rho_{13}$ and $\rho_{23}$ axes; they intersect at the origin $\model_{\text{sing}}$. See also \citet[Figure 3]{evans2018model}.}
\label{fig:geometry}
\end{figure}

$\mathcal{M}_0$ and $\mathcal{M}_1$ are non-nested and they further intersect at the origin
\begin{equation}
\mathcal{M}_{\text{sing}}: \{\sigma_{12} = \sigma_{13} = \sigma_{23} = 0\},
\end{equation}
which is a \emph{singularity} within $\model_0 \cap \model_1$ that corresponds to diagonal covariances. At $\model_{\text{sing}}$ the likelihood ratio statistic is not regular in the sense that the tangent cones (linear approximations to the parameter space; see \citet[Chap. 4.6]{bertsekas2003convex}) of the two models coincide. 
As pointed out by \citet{evans2018model}, we will see that the equivalence of local geometry between the two models presents a challenge for model selection. 
It is also worth mentioning that, in the setting of nested model selection, the behavior of the likelihood ratio of testing $\mathcal{M}_0 \cap \mathcal{M}_1$ against a saturated model, especially at the singularity, has been studied by \citet{drton2006algebraic, drton2007algebraic, drton2009likelihood}.

\subsection*{Organization} The paper is organized as follows. In \cref{sec:mle}, we derive the maximum likelihood estimates under the two types of independence models, and obtain the loglikelihood ratio statistic in a closed form. In \cref{sec:minimax}, we characterize the information-theoretic limit to distinguishing the two models, and outline two regimes on the boundary of distinguishability. Then in \cref{sec:local-asymp}, we consider local alternative sequences in the two aforementioned regimes and establish the asymptotic distribution of the loglikelihood ratio. \cref{sec:limit-exp} provides a geometric perspective in terms of limit experiments. We then deal with non-uniformity issue of asymptotic distributions in \cref{sec:envelope} by introducing a family of envelope distributions. Next in \cref{sec:procedure} we propose model selection procedures with a uniform error guarantee. In \cref{sec:simu}, we compare the performance of several methods through simulation studies. We present a realistic example in \cref{sec:real-data} on inferring the American occupational structure. Finally some discussions are given in \cref{sec:discussion}.

\section{Maximum likelihood in a trivariate Gaussian model} \label{sec:mle}
The log-likelihood of a Gaussian graphical model under sample size $n$ \citep[Chap. 5]{lauritzen1996graphical} is
\begin{equation}
\ell_n(\Sigma) = \frac{n}{2} (- \log |\Sigma| - \Tr(S_n \Sigma^{-1})), \label{eqs:ll}
\end{equation}
where $S_n$ is the sample covariance computed with respect to mean zero (i.e., the scatter matrix divided by $n$). A model can be scored by its log-likelihood maximized within the model contrasted against the \textit{saturated} model
\begin{equation}
\lambda_n^{(i)} := 2\left (\sup_{\Sigma \in \Theta} \ell_n(\Sigma) - \sup_{\Sigma \in \Theta_i} \ell_n(\Sigma) \right ) \geq 0
\end{equation}
for $i=0,1$, which is the quantity considered in nested model selection. 
The saturated model attains maximal likelihood when $\Sigma = S_n$, yielding
\begin{equation*}
\begin{split}
\ell_{n}^{\text{sat}} &:= \sup_{\Sigma \in \Theta} \ell_n(\Sigma)\\
&= -\frac{n}{2} \left(\log \left({s}_{11} {s}_{22} {s}_{33} + 2 {s}_{12} {s}_{23} {s}_{13}- {s}_{11} {s}_{23}^2- {s}_{13}^2 {s}_{22} - {s}_{12}^2 {s}_{33}\right)+3 \right ).
\end{split}
\end{equation*}

To contrast $\mathcal{M}_0$ and $\mathcal{M}_1$, we instead consider 
\begin{equation}
\begin{split}
\lambda_n^{(0:1)} &:= \lambda_n^{(1)} - \lambda_n^{(0)} \\
& = 2 \left ( \sup_{\Sigma \in \Theta_0} \ell_n(\Sigma) - \sup_{\Sigma \in \Theta_1} \ell_n(\Sigma) \right ) = 2 \left(\ell_n(\hat{\Sigma}_n^{(0)}) - \ell_n(\hat{\Sigma}_n^{(1)}) \right),
\end{split}
\end{equation}
where $\hat{\Sigma}_n^{(0)}$ and $\hat{\Sigma}_n^{(1)}$ are MLEs within the two models. Intuitively, a positive value of $\lambda_n^{(0:1)}$ prefers $\mathcal{M}_0$, and a negative value prefers $\mathcal{M}_1$. 

\subsection{MLE within $\mathcal{M}_0$} By $X_1 \indep X_2$, we can factorize the likelihood of $\mathcal{M}_0$ as 
\begin{equation*}
\begin{split}
p(X_1, X_2, X_3) &= p(X_1) p(X_2) p(X_3 \mid X_1, X_2) \\
&= \N(X_1; 0, \sigma_{11}) \N(X_2; 0, \sigma_{22}) \N(X_3; \beta_{32 \cdot 1} X_1 + \beta_{31 \cdot 2} X_2, \sigma_{33 \cdot 12}).
\end{split}
\end{equation*}
where the parameters $\sigma_{11}, \sigma_{22}, \beta_{32 \cdot 1}, \beta_{31 \cdot 2}, \sigma_{33 \cdot 12}$ are variation independent \citep[Chap. 10.2]{barndorff2014information}. The MLEs for them are given by 
\begin{equation}
\hat{\sigma}_{11}^{(0)} = s_{11}, \quad \hat{\sigma}_{22}^{(0)} = s_{22}, \label{eqs:mle-0-i}
\end{equation}
\begin{equation*}
\quad \hat{\beta}_{32 \cdot 1}^{(0)} = \frac{{s}_{22} {s}_{13}-{s}_{12} {s}_{23}}{{s}_{11} {s}_{22}-{s}_{12}^2}, \quad \hat{\beta}_{31 \cdot 2}^{(0)} = \frac{s_{11} s_{23}-s_{12} s_{13}}{s_{11} s_{22}-s_{12}^2},
\end{equation*}
and
\begin{equation*}
\hat{\sigma}_{33 \cdot 12}^{(0)} = s_{33} - \frac{s_{22} s_{13}^2-2 s_{12} s_{23} s_{13}+s_{11} s_{23}^2}{s_{11} s_{22} - s_{12}^2}.
\end{equation*}
Mapping them back to the original parameters via relations $\beta_{32 \cdot 1} = \sigma_{13} / \sigma_{11}$, $\beta_{31 \cdot 2} = \sigma_{23} / \sigma_{22}$ and $\sigma_{33 \cdot 12} =  \sigma_{33} - \sigma_{13}^2 / \sigma_{11} - \sigma_{23}^2 / \sigma_{22}$, in addition to \cref{eqs:mle-0-i} we have the MLEs as
\begin{equation}
\hat{\sigma}_{13}^{(0)} = \frac{{s}_{11} \left({s}_{22} {s}_{13}-{s}_{12} {s}_{23}\right)}{{s}_{11} {s}_{22}-{s}_{12}^2}, \quad {\hat{\sigma}}_{23}^{(0)} = \frac{s_{22} \left(s_{11} s_{23}-s_{12} s_{13}\right)}{s_{11} s_{22}-s_{12}^2} \label{eqs:mle-0-ii}
\end{equation}
and 
\begin{equation}
\hat{\sigma}_{33}^{(0)} = {s}_{33}-\frac{2 {s}_{12} \left({s}_{12} {s}_{13}-{s}_{11} {s}_{23}\right) \left({s}_{12} {s}_{23}-{s}_{13} {s}_{22}\right)}{\left({s}_{11} {s}_{22} - {s}_{12}^2\right){}^2}. \label{eqs:mle-0-iii}
\end{equation}

This derivation is essentially the same as executing the iterative conditional fitting algorithm of \citet{chaudhuri2007estimation} in the order of $X_1,X_2,X_3$. 
Plugging \cref{eqs:mle-0-i,eqs:mle-0-ii,eqs:mle-0-iii} into \cref{eqs:ll}, we have the following \emph{closed-form} expression of maximized log-likelihood of $\model_0$
\begin{equation}
\ell_n^{(0)} = -\frac{n}{2} \left[ \log \left({s}_{11} {s}_{22} \left(\frac{{s}_{22} \hat{s}_{13}^2-2 {s}_{12} {s}_{23} {s}_{13}+{s}_{11} {s}_{23}^2}{{s}_{12}^2-{s}_{11} {s}_{22}}+{s}_{33}\right)\right) + 3 \right]. \label{eqs:ll-0}
\end{equation}

\subsection{MLE within $\mathcal{M}_1$} The MLE within $\mathcal{M}_1$ in the covariance parametrization is simpler. By writing $\sigma_{12} = \sigma_{13} \sigma_{23} / \sigma_{33}$ and simplifying the score condition, we obtain
\begin{equation}
\hat{\sigma}_{11}^{(1)} = s_{11}, \quad  \hat{\sigma}_{22}^{(1)} = s_{22}, \quad  \hat{\sigma}_{33}^{(1)} = s_{33}, \quad  \hat{\sigma}^{(1)}_{13} = s_{13}, \quad \hat{\sigma}_{23}^{(1)} = s_{23}, \label{eqs:mle-1}
\end{equation}
all of which are their sample counterparts. Plugging into \cref{eqs:ll}, we have
\begin{equation}
\ell_{n}^{(1)} = -\frac{n}{2} \left[ \log \left(\frac{\left(s_{13}^2-s_{11} s_{33}\right) \left(s_{23}^2-s_{22} s_{33}\right)}{s_{33}}\right) + 3 \right]. \label{eqs:ll-1}
\end{equation}

\subsection{Likelihood ratio} Finally, $\mathcal{M}_0$ and $\mathcal{M}_1$ are contrasted with
\begin{multline}
\lambda_n^{(0:1)} = 2(\ell_n^{(0)} - \ell_n^{(1)})  = n \log \left(\frac{\left(s_{13}^2-s_{11} s_{33}\right) \left(s_{23}^2-s_{22} s_{33}\right)}{s_{33}}\right)- \\
n\log \left(s_{11} s_{22} \left(\frac{s_{22} s_{13}^2-2 s_{12} s_{23} s_{13}+s_{11} s_{23}^2}{s_{12}^2-s_{11} s_{22}}+s_{33}\right)\right). \label{eqs:lambda-01}
\end{multline}
 \section{Optimal error} \label{sec:minimax}
We study the information-theoretic limit to distinguishing the two models. Specifically, consider two sequences of sampling distributions --- one within $\model_0$ and the other within $\model_1$, as they approach the same limit in $\model_0 \cap \model_1$. Let $P_n$ be the sequence in $\model_0$ under covariance $\Sigma_n^{(0)} \in \model_0 \setminus \model_1$, and let $Q_n$ be the sequence in $\model_1$ under covariance $\Sigma_n^{(1)} \in \model_1 \setminus \model_0$. Further, let $P_n^{n}$ and $Q_n^n$ be the product measures of $n$ independent copies of $P_n$ and $Q_n$ respectively.

The fundamental limit to distinguishing two distributions $P$ and $Q$ is characterized by their total variation distance $\TV(P, Q) := \sup_{A} \{P(A) - Q(A)\}$. We have the following classical result on testing two simple hypotheses, where the minimum total error is achieved by the likelihood ratio test.

\begin{lemma}[Theorem 13.1.1 of \citet{lehmann2006testing}] \label{lem:TV}
For testing $H_0: X \sim P$ versus $H_1: X \sim Q$, the minimum sum of type-I and type-II errors is $1 - \TV(P, Q)$. 
\end{lemma}

The optimal error above does not permit a tractable formula. The analysis for a product measure is more tractable in terms of the Hellinger squared distance $H^2(P,Q) :=  (1/2) \int (p^{1/2} - q^{1/2})^2 \dd \mu $, for which it holds that
\begin{equation}
H^2(P_n^n, Q_n^n) = 1 - \left\{1 - H^2(P_n, Q_n) \right\}^n. \label{eqs:hellinger-tensor}
\end{equation}
The total variation is related to Hellinger by Le~Cam's inequality \citep[Lemma 2.3]{Tsybakov2009book}
\begin{equation}
H^2(P_n^n, Q_n^n) \leq \TV(P_n^n, Q_n^n) \leq H(P_n^n, Q_n^n) \left \{ 2 - H^2(P_n^n, Q_n^n) \right \}^{1/2}. \label{eqs:sandwich}
\end{equation}

\begin{lemma}[see also Theorem 13.1.3 of \citet{lehmann2006testing}] \label{lem:stabilized-error} It holds that
\begin{equation*}
1 - \TV(P_n^n, Q_n^n) \rightarrow \begin{cases}
0, &\quad H^2(P_n, Q_n) = \omega(n^{-1}) \\
1, &\quad H^2(P_n, Q_n) = o(n^{-1})
\end{cases}.
\end{equation*}
And when $n H^2(P_n, Q_n) \rightarrow h > 0$, it holds that
\begin{equation*}
\begin{split}
0 < 1 - \{1 - \exp(-2h) \}^{1/2} &\leq \liminf_{n \rightarrow \infty} \left\{ 1 - \TV(P_n^n, Q_n^n) \right\}\\
&\leq \limsup_{n \rightarrow \infty} \left\{ 1 - \TV(P_n^n, Q_n^n) \right\} \leq \exp(-h) < 1.
\end{split}
\end{equation*}
\end{lemma}
\begin{proof}
Using \cref{eqs:hellinger-tensor} and \cref{eqs:sandwich}, we have
\begin{equation*}
\begin{split}
1 - \TV(P_n^n, Q_n^n) \leq 1 - H^2(P_n^n, Q_n^n) &= \left\{1 - H^2(P_n, Q_n) \right\}^n \\
&= \exp \left [n \log \left\{1 - H^2(P_n,Q_n) \right \} \right ].
\end{split}
\end{equation*}
It follows that
\begin{equation*}
\begin{split}
\limsup_{n \rightarrow \infty} \left\{ 1 - \TV(P_n^n, Q_n^n) \right\} &\leq \limsup_{n \rightarrow \infty} \exp \left [n \log \left\{1 - H^2(P_n,Q_n) \right \} \right ] \\
& = \begin{cases} 0, &\quad H^2(P_n, Q_n) = \omega(n^{-1}) \\
\exp(-h), &\quad n H^2(P_n, Q_n) \rightarrow h > 0
\end{cases}.
\end{split}
\end{equation*}
Similarly, we also have
\begin{equation*}
\begin{split}
\liminf_{n \rightarrow \infty} \left\{ 1 - \TV(P_n^n, Q_n^n) \right\} & \geq \liminf_{n \rightarrow \infty} 1 - H(P_n^n, Q_n^n) \left \{ 2 - H^2(P_n^n, Q_n^n) \right \}^{1/2}  \\
& = \begin{cases} 1, & \quad  H^2(P_n, Q_n) = o(n^{-1}) \\
1 - \{1 - \exp(-2h) \}^{1/2}, &\quad  n H^2(P_n, Q_n) \rightarrow h > 0
\end{cases}.
\end{split}
\end{equation*}
The proof is finished by combining the previous two displays with the fact that $\liminf_{n} \left\{ 1 - \TV(P_n^n, Q_n^n) \right\} \leq \limsup_{n} \left\{ 1 - \TV(P_n^n, Q_n^n) \right\}$ and noting $\TV \in [0,1]$.
\end{proof}

\begin{corollary}
Under $n H^2(P_n, Q_n) \rightarrow h > 0$, the optimal power of an asymptotic $\alpha$-level procedure satisfies
\begin{equation}
1 - \exp(-h) \leq \text{optimal asymptotic power} \leq \alpha + \{1 - \exp(-2h)\}^{1/2}. \label{eqs:bound-power}
\end{equation}
\end{corollary}
\begin{proof}
This directly follows from \cref{lem:stabilized-error} since $(1 - \text{optimal power}) + \text{type-I error} = 1 -  \TV(P_n^n, Q_n^n) $ for type-I error asymptotically between 0 and $\alpha$, and then passing to the limit.  
\end{proof}

By \cref{lem:stabilized-error}, the asymptotic error converges to zero (exponentially fast) if $P_n$ and $Q_n$ are separated by a distance that is decreasing more slowly than rate $n^{-1/2}$. For example, when $P_n = P$, $Q_n = Q$ are \emph{fixed} distributions from which we observe $n$ independent samples, that is, when $P_n^n = P^{n}$ and $Q_n^n = Q^n$. 
The analysis above shows that the ability to differentiate $P_n$ and $Q_n$ based on $n$ samples depends on the distance between $P_n$ and $Q_n$. The consideration of $P_n$ and $Q_n$ as $n \rightarrow \infty$ is necessitated by the development of asymptotic results that are applicable in a specific analysis with a fixed $n$. 
In particular, here we want to investigate what happens when the sample size is small compared to the signal strength, or equivalently, when signal strength is weak under a given sample size. This is modeled by the regime that yields a non-trivial optimal error strictly between 0 and 1. By \cref{lem:stabilized-error}, we need to choose sequences $\Sigma_n^{(0)}$ and $\Sigma_n^{(1)}$ such that $H^2 \left(P_{\Sigma_n^{(0)}}, P_{\Sigma_n^{(1)}} \right) \asymp n^{-1}$. More specifically, we choose $\Sigma_n^{(1)}$ that is the most difficult to distinguish from $\Sigma_n^{(0)}$. That is, we choose $\Sigma_n^{(1)}$ to minimize $\KL \left(P_{\Sigma_n^{(0)}} \| P_{\Sigma_n^{(1)}} \right)$, i.e., $\Sigma_n^{(1)}$ is the MLE projection of $\Sigma_n^{(0)}$ in $\model_1$ by \cref{eqs:mle-1}. The two sequences take the form of
\begin{equation}
\begin{split}
\Sigma_n^{(0)} &= \begin{pmatrix}
\sigma_{11,n} & 0 & \rho_{13,n} \sqrt{\sigma_{11,n} \sigma_{33,n}} \\
0  & \sigma_{22,n} & \rho_{23,n} \sqrt{\sigma_{11,n} \sigma_{33,n}} \\
\rho_{13,n} \sqrt{\sigma_{11,n} \sigma_{33,n}}  & \rho_{23,n} \sqrt{\sigma_{11,n} \sigma_{33,n}} & \sigma_{33,n}
\end{pmatrix},\\
\Sigma_n^{(1)} &= \begin{pmatrix}
\sigma_{11,n} & \rho_{13,n}\rho_{23,n} \sqrt{\sigma_{11,n} \sigma_{22,n}} & \rho_{13,n} \sqrt{\sigma_{11,n} \sigma_{33,n}} \\
\rho_{13,n}\rho_{23,n} \sqrt{\sigma_{11,n} \sigma_{22,n}}  & \sigma_{22,n} & \rho_{23,n} \sqrt{\sigma_{22,n} \sigma_{33,n}} \\
\rho_{13,n} \sqrt{\sigma_{11,n} \sigma_{33,n}}  & \rho_{23,n} \sqrt{\sigma_{22,n} \sigma_{33,n}} & \sigma_{33,n}\end{pmatrix}.
\end{split}
\end{equation}
Both of them converge to $\Sigma^{\ast} \in \model_0 \cap \model_1$ as $n \rightarrow \infty$. 
We assume the variances $\sigma_{ii,n} \rightarrow \sigma_{ii} > 0$ for $i=1,2,3$. For $H^2 \left(P_{\Sigma_n^{(0)}}, P_{\Sigma_n^{(1)}} \right) \rightarrow 0$, it is necessary that either (or both) $\rho_{13,n}$ and $\rho_{23,n}$ converges to zero. The squared Hellinger distance is calculated as 
\begin{align}
H^2 \left(P_{\Sigma_n^{(0)}}, P_{\Sigma_n^{(1)}} \right) &= 1 - \frac{ |{\Sigma}_n^{(0)}|^{1/4} |{\Sigma}_n^{(1)}|^{1/4}}{|({\Sigma}_n^{(0)} + {\Sigma}_n^{(1)})/2|^{1/2}} \nonumber \\
& =\begin{cases}
\rho_{13,n}^2 \rho_{23,n}^2 / 8 + O(\rho_{13,n}^4 + \rho_{23,n}^4), \quad &\rho_{13,n}, \rho_{23,n} \rightarrow 0 \\
\rho_{23}^2 (1 - \rho_{23}^2)^{-1} \rho_{13,n}^2 / 8 + O(\rho_{13,n}^4), \quad &\rho_{13,n} \rightarrow 0, \rho_{23,n} \rightarrow \rho_{23} \neq 0 \\
\rho_{13}^2 (1 - \rho_{13}^2)^{-1} \rho_{23,n}^2 / 8 + O(\rho_{23,n}^4), \quad &\rho_{23,n} \rightarrow 0, \rho_{13,n} \rightarrow \rho_{13} \neq 0
\end{cases}.
\label{eqs:minimax-helliger}
\end{align}

The calculation reveals that $H^2(P_{\Sigma_n^{(0)}}, Q_{\Sigma_n^{(1)}}) \asymp 1/n$ if and only if $\rho_{13,n} \rho_{23,n} \asymp n^{-1/2}$. This entails two distinct regimes. 

\begin{description}
\item[The weak-strong regime] Between $\rho_{13,n}$ and $\rho_{23,n}$, one (the weak edge) converges to zero at $n^{-1/2}$ rate, and the other (the strong edge) converges to a non-zero limit $\rho \in (-1,1)$. The limiting model is on $\model_0 \cap \model_1 \setminus \model_{\text{sing}}$, namely one of the axes excluding the origin in \cref{fig:geometry}.
\item[The weak-weak regime] $\rho_{13,n}, \rho_{23,n} \rightarrow 0$ and $\sqrt{n} \rho_{13,n} \rho_{23,n} \rightarrow \delta \neq 0$. The limiting model is on $\model_{\text{sing}}$, namely the origin in \cref{fig:geometry}. 
\end{description}

\begin{remark}
The result can be rephrased as the sample size required to distinguish $\model_0$ and $\model_1$.
Consider distinguishing $\model_0$ and $\model_1$ in a Euclidean $m^{-1/2}$-neighborhood of $\Sigma^{\ast} \in \model_0 \cap \model_1$ as $m \rightarrow \infty$. The sample size required is $m^2$ if $\Sigma^{\ast} \in \model_{\text{sing}}$, and $m$ if $\Sigma^{\ast} \notin \model_{\text{sing}}$. This phenomenon is described by \citet{evans2018model} in terms of equivalence of local geometry. $\model_0$ and $\model_1$ are 1-equivalent at $\Sigma^{\ast} \in \model_{\text{sing}}$ in the sense that their tangent cones coincide; and they are 1-near-equivalent at $\Sigma^{\ast} \notin \model_{\text{sing}}$ in the sense that they have distinct tangent cones. See \citet[Theorem 2.8]{evans2018model}.
\end{remark}

\begin{proposition} \label{prop:complexity}
In testing $\mathcal{M}_0$ versus $\mathcal{M}_1$, the sample complexity required is
\begin{equation}
\begin{cases}
n = \omega \left(\frac{1}{\rho_{13}^2 \rho_{23}^2} \right), &\quad \text{for consistent model selection}\\
n \asymp \left(\frac{1}{\rho_{13}^2 \rho_{23}^2} \right), &\quad \text{for asymptotic total error $\in (0,1)$}
\end{cases}.
\end{equation}
\end{proposition}

 \section{Local asymptotics} \label{sec:local-asymp}
In this section, we analyze the asymptotic distribution of the log-likelihood ratio statistic $\lambda_n^{(0:1)}=2(\ell_n^{(0)} - \ell_n^{(1)})$ under the two regimes outlined earlier. 

\subsection{The weak-strong regime}
Without loss of generality, we choose $\rho_{13,n} = \gamma / \sqrt{n}$ as the weak edge and $\rho_{23,n} \rightarrow \rho \neq 0$ as the strong edge. $\gamma \in \mathbb{R}$ characterizes the size of the local asymptotic, and is also referred to as a \emph{local parameter} \citep[Chapter 9]{van2000asymptotic}. We consider asymptotics under local sequences $\Sigma_n^{(0)}$ and $\Sigma_n^{(1)}$ approaching the limiting covariance 
\begin{equation}
\Sigma^{\ast} = \begin{pmatrix}
\sigma_{11} & 0 & 0 \\
0 & \sigma_{22} & \rho \sqrt{\sigma_{22} \sigma_{33}}  \\
0 & \rho \sqrt{\sigma_{22} \sigma_{33}}  & \sigma_{33}
\end{pmatrix} \in \model_0 \cap \model_1 \setminus \model_{\text{sing}}.
\end{equation}
We consider the following local alternatives of size $\gamma$ on the correlation scale. Again, $\Sigma_n^{(1)}$ is the KL-projection (i.e., MLE-projection) of $\Sigma_n^{(0)}$. 
\begin{equation}
\Sigma^{(0)}_n = \begin{pmatrix} 
\sigma_{11} & 0 & \gamma \sqrt{\sigma_{11} \sigma_{33}} / \sqrt{n} \\
0 & \sigma_{22} & \rho \sqrt{\sigma_{22} \sigma_{33}} \\
\gamma \sqrt{\sigma_{11} \sigma_{33}} / \sqrt{n} & \rho \sqrt{\sigma_{22} \sigma_{33}} & \sigma_{33}
\end{pmatrix} \in \model_0 \setminus \model_1, \label{eqs:sigma-0-ws}
\end{equation}
\begin{equation}
\Sigma^{(1)}_n = \begin{pmatrix} 
\sigma_{11} & \gamma \rho \sqrt{\sigma_{11} \sigma_{22}} / \sqrt{n} & \gamma \sqrt{\sigma_{11} \sigma_{33}} / \sqrt{n} \\
\gamma \rho \sqrt{\sigma_{11} \sigma_{22}} / \sqrt{n} & \sigma_{22} & \rho \sqrt{\sigma_{22} \sigma_{33}} \\
\gamma \sqrt{\sigma_{11} \sigma_{33}} / \sqrt{n} & \rho \sqrt{\sigma_{22} \sigma_{33}} & \sigma_{33}
\end{pmatrix} \in \model_1 \setminus \model_0. \label{eqs:sigma-1-ws}
\end{equation}

At the limit, both models are correct (intersection); see \Cref{fig:local-schematic}. However, the sequence approaches the limit \textit{only} on one of the models, and the size of violation of the other model is $|\gamma|n^{-1/2}$. To ensure positive definiteness, we require $|\rho| < 1$.

\begin{proposition} \label{prop:asymp-weak-strong}
Under local alternative $\Sigma_n^{(0)}$,
\begin{equation}
\lambda_{n}^{(0:1)} \distconvto \rho \left [ \left(Z_1 + \frac{\gamma}{\sqrt{2 (1 - \rho)}} \right)^2 - \left(Z_2 + \frac{\gamma}{\sqrt{2(1+\rho)}} \right)^2 \right ]; \label{eqs:asymp-dist-01}
\end{equation}
and under local alternative $\Sigma_n^{(1)}$,
\begin{equation}
\lambda_{n}^{(0:1)} \distconvto \rho \left [ \left(Z_1 + \gamma \sqrt{\frac{1-\rho}{2}} \right)^2 - \left(Z_2 + \gamma \sqrt{\frac{1 + \rho}{2}} \right)^2 \right ],\label{eqs:asymp-dist-10}
\end{equation}
where $Z_1, Z_2$ are two independent standard normal variables. 
\end{proposition}

We leave the proof to the next section, where we will present a geometric interpretation of the asymptotic distributions. Alternatively, the distribution can be derived by a change of measure with Le Cam's third lemma; see \citet[Example 6.7]{van2000asymptotic}. 

Asymptotically the log-likelihood ratio statistic is distributed as a scaled difference of two independent non-central $\chi^2_1$ variables, with non-centralities scaled by $\gamma$ and weighted by $\rho$ differently, depending on the true model. Note that the distribution only depends on the absolute values of $\gamma$ and $\rho$. The asymptotic distributions under the two types of sequences (truths) are visualized in \cref{fig:asymp-dist-ws}. We can see that the mean is positive under $\model_0 \setminus \model_1$ and negative under $\model_1 \setminus \model_0$. However, a pair of these distributions are not symmetric to each other in terms of shape. They are further separated apart (more easily distinguished) as $|\gamma|$ or $|\rho|$ becomes bigger, and only become identical (distributed as $\rho (Z_1^2 - Z_2^2)$) when $\gamma \rightarrow 0$.

\begin{figure}[!ht]
\includegraphics[width=1.\textwidth]{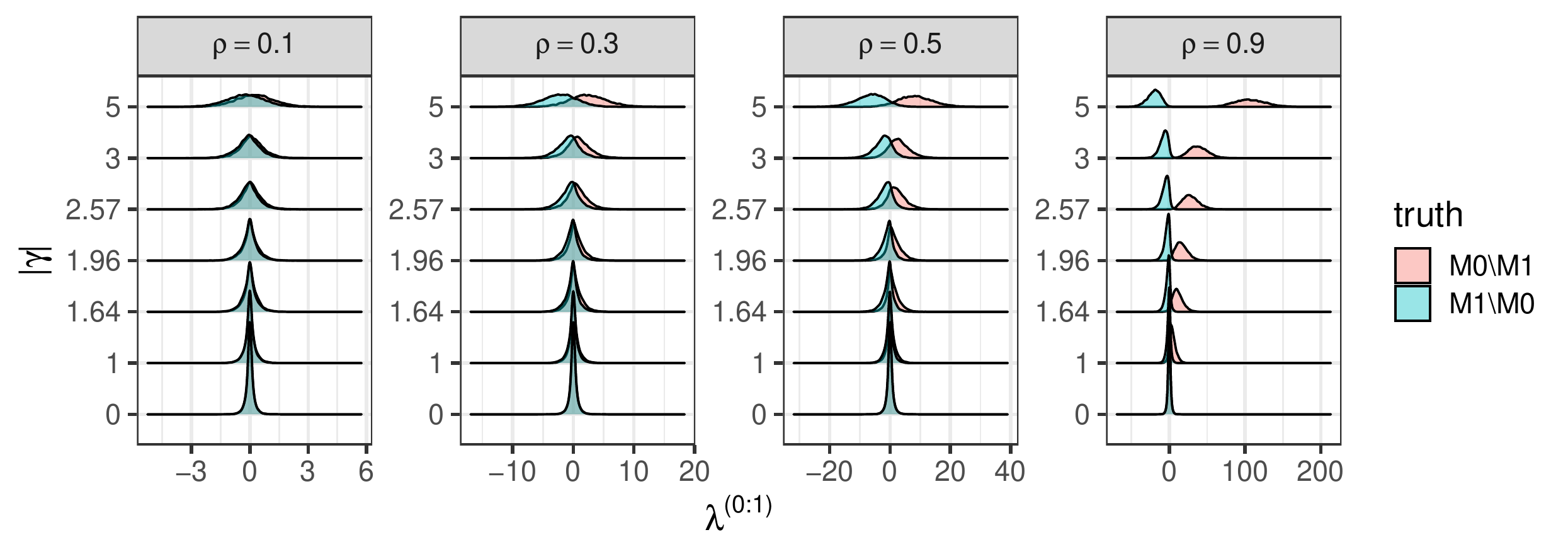}
\caption{Asymptotic distributions of $\lambda_n^{(0:1)}$ under $\Sigma_n^{(0)} \in \model_0 \setminus \model_1$ and $\Sigma_n^{(1)} \in \model_1 \setminus \model_0$ in the weak-strong regime.}
\label{fig:asymp-dist-ws}
\end{figure}

\begin{remark}
The models are locally asymptotically normal at $\Sigma^{\ast} \notin \model_{\text{sing}}$. By regularity, replacing the constant elements in \cref{eqs:sigma-0-ws,eqs:sigma-1-ws} with sequences in $n$ that converge to the corresponding limits does not alter the asymptotic distribution of $\lambda_n^{(0:1)}$.
\end{remark}

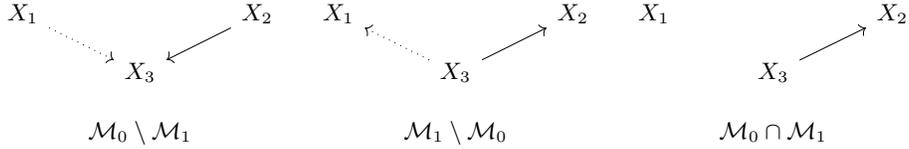
\begin{figure}[!ht]
\begin{tikzcd}[sep=small]
X_1 \arrow[rd, dotted, rightarrow] &  & X_2 \arrow[ld, rightarrow] & X_1 &  & X_2 & X_1 &  & X_2 \\
 & X_3 &  &  & X_3 \arrow[lu, dotted] \arrow[ru] &  &  & X_3 \arrow[ru] &  \\
 & \mathcal{M}_0 \setminus \mathcal{M}_1&  &  & \mathcal{M}_1 \setminus \mathcal{M}_0 &  &  & \mathcal{M}_0 \cap \mathcal{M}_1 & 
\end{tikzcd}
\caption{Two types of local sequences and their common limit.}
\label{fig:local-schematic}
\end{figure}

\subsection{The weak-weak regime} \label{ssec:asymp-ww}
Now we study the asymptotic under $\rho_{13,n}, \rho_{23,n} \rightarrow 0$ and $\sqrt{n} \rho_{13,n} \rho_{23,n} \rightarrow \delta$. The limiting covariance is $\Sigma^{\ast} = \diag(\sigma_{11}, \sigma_{22}, \sigma_{33}) \in \model_{\text{sing}}$, towards which we consider two local sequences
\begin{equation} \label{eqs:sigma-0-ww}
\Sigma_n^{(0)} = \begin{pmatrix}
\sigma_{11} & 0 & \rho_{13,n} \sqrt{\sigma_{11} \sigma_{33}} \\
0  & \sigma_{22} & \rho_{23,n} \sqrt{\sigma_{11} \sigma_{33}} \\
\rho_{13,n} \sqrt{\sigma_{11} \sigma_{33}}  & \rho_{23,n} \sqrt{\sigma_{11} \sigma_{33}} & \sigma_{33}
\end{pmatrix} \in \model_0 \setminus \model_1
\end{equation}
and 
\begin{equation} \label{eqs:sigma-1-ww}
\Sigma_n^{(1)} = \begin{pmatrix}
\sigma_{11} & \rho_{13,n}\rho_{23,n} \sqrt{\sigma_{11} \sigma_{22}} & \rho_{13,n} \sqrt{\sigma_{11} \sigma_{33}} \\
\rho_{13,n}\rho_{23,n} \sqrt{\sigma_{11} \sigma_{22}}  & \sigma_{22} & \rho_{23,n} \sqrt{\sigma_{22} \sigma_{33}} \\
\rho_{13,n} \sqrt{\sigma_{11,n} \sigma_{33}}  & \rho_{23,n} \sqrt{\sigma_{22} \sigma_{33}} & \sigma_{33}\end{pmatrix} \in \model_1 \setminus \model_0.
\end{equation}

\begin{proposition}\label{prop:asymp-ww}
Given $\rho_{13,n} \rho_{23,n} = \delta n^{-1/2} + o(n^{-1/2})$ for $\delta \neq 0$ and $\rho_{13,n
}, \rho_{23,n} \rightarrow 0$. Under $\Sigma_{n}^{(i)} \in \model_i \setminus \model_{1-i} $ for $i=0,1$, we have
\begin{equation}
\lambda_{n}^{(0:1)} \distconvto \delta (2 Z + (-1)^{i} \delta) =_{d} \N \left((-1)^{i} \delta^2, (2 \delta)^2 \right).
\end{equation}
\end{proposition}

The limit is a centered Gaussian shifted and then scaled by $\delta$. Plots for a few values of $\delta$ are given by \cref{fig:row-gaussians}.

\begin{figure}[!ht]
\includegraphics[width=1.\textwidth]{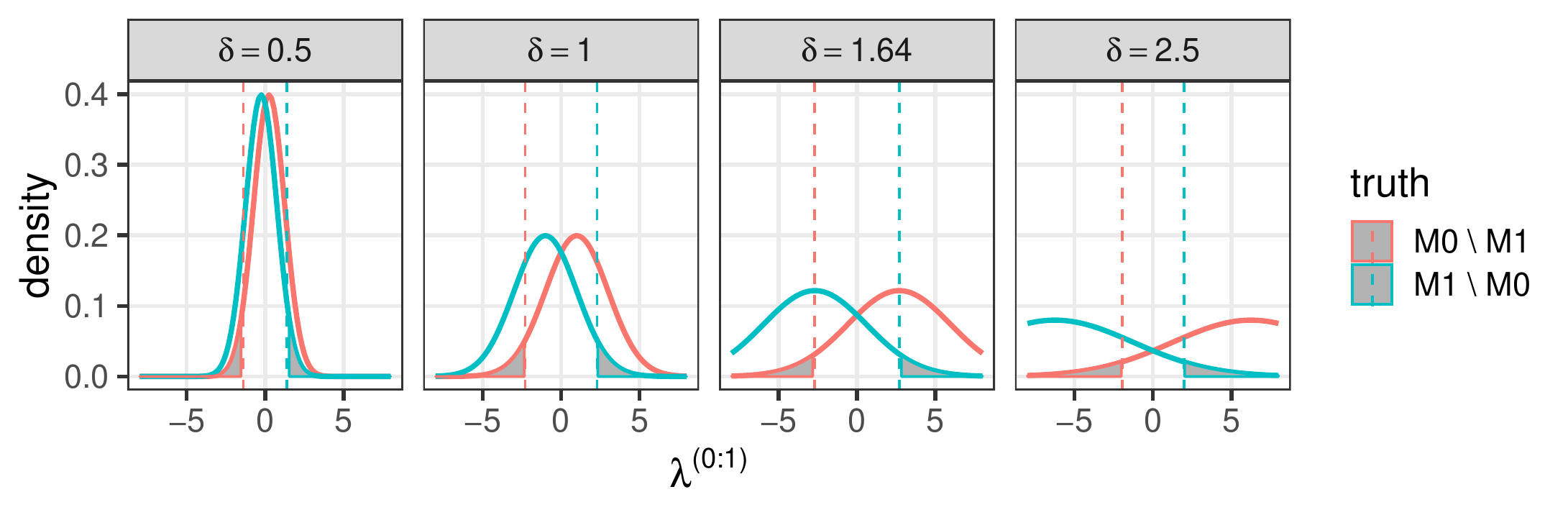}
\caption{Asymptotic distribution of $\lambda_n^{(0:1)}$ in the weak-weak regime under $\model_0 \setminus \model_1$ (red) and $\model_1 \setminus \model_0$ (blue). The vertical lines and shaded areas correspond to $95\%$ upper/lower quantiles.}
\label{fig:row-gaussians}
\end{figure}

\begin{proof}[Proof of \cref{prop:asymp-ww}]
By the coincidence of tangent cones of the two models in this regime, the distribution cannot be obtained from local asymptotic normality or contiguity relative to a tensorized static law. Instead, we perform a direct calculation. For convenience, we assume the form of (sub)sequences of $\rho_{13,n}$ and $\rho_{23,n}$ as
\begin{equation*}
\rho_{13,n} = \eta n^{-a}, \quad \rho_{23,n} = \tau n^{-(1/2 - a)}
\end{equation*}
for $a \in (0, 1/2)$ and $\eta \tau = \delta$. 
We perform a manual change of measure by relating the law under $P_{\Sigma_n^{(i)}}$ to that under $P_{I}$, which is iid sampling of $\N(0, I)$. Under sample size $n$, suppose $\Omega_n$ is the sample covariance under $\N(0, I)$. Now suppose $S_n^{(i)}$ is the sample covariance under $P_{\Sigma_n^{(i)}}$ for $i=0,1$. Then it holds that 
\begin{equation}
S_n^{(i)} =_{d} L_n^{(i)} \Omega_n L_n^{{(i)}\T},
\end{equation}
for some $L_n^{(i)}$ such that $\Sigma_n^{(i)} = L_n^{(i)} L_n^{^{(i)} \T}$.
Here we choose them as the Cholesky decompositions
\begin{equation*}
L_n^{(0)} = \left(
\begin{array}{ccc}
 \sqrt{\sigma _{11}} & 0 & 0 \\
 0 & \sqrt{\sigma _{22}} & 0 \\
 n^{-a} \eta  \sqrt{\sigma _{33}} & n^{a-\frac{1}{2}} \tau  \sqrt{\sigma _{33}} & \sqrt{\left(-\eta^2 n^{-2 a}-\tau ^2 n^{2 a-1}+1\right) \sigma _{33}} \\
\end{array}
\right)
\end{equation*}
and 
\begin{equation*}
L_n^{(1)} = \left(
\begin{array}{ccc}
 \sqrt{\sigma _{11}} & 0 & 0 \\
 \eta  \tau  \sqrt{\frac{\sigma _{22}}{n}} & \sqrt{\frac{\left(n-\eta ^2 \tau ^2\right) \sigma _{22}}{n}} & 0 \\
 n^{-a} \eta  \sqrt{\sigma _{33}} & n^{-a} \left(n^{2 a}-\gamma ^2\right) \tau  \sqrt{\frac{\sigma _{33}}{n-\eta ^2 \tau ^2}} & \sqrt{\frac{n^{-2 a} \left(n^{2 a}-\eta ^2\right) \left(n^{2 a} \tau ^2-n\right) \sigma _{33}}{\eta ^2 \tau ^2-n}} \\
\end{array}
\right).
\end{equation*}

By the central limit theorem, we have
\begin{equation}
\sqrt{n} (\Omega_n - I) \distconvto W
\end{equation}
for $W$ a $3 \times 3$ matrix of joint Gaussian variables whose covariance is determined by the Isserlis matrix. The asymptotic distribution of $\lambda_n^{(0:1)}$ can be obtained by substituting
\begin{equation}
S_n^{(i)} = L_n^{(i)} \left(I + n^{-1/2} W + o_p(n^{-1/2}) \right) L_n^{(i)\T}
\end{equation}
into the closed-form expression of \cref{eqs:lambda-01} and simplifying. We have under $\Sigma_n^{(0)}$
\begin{equation}
\lambda_n^{(0:1)} = \gamma \tau (\gamma \tau - 2 w_{12}) + o_p(1),
\end{equation}
and under $\bm{\Sigma}_n^{(1)}$
\begin{equation}
\lambda_n^{(0:1)} = -\gamma \tau (\gamma \tau + 2 w_{12}) + o_p(1).
\end{equation}
The result is immediate from $w_{12} \sim \N(0,1)$ and $\gamma \tau = \delta$.
\end{proof}

\begin{remark}
The Gaussian asymptotic in \cref{prop:asymp-ww} does \emph{not} depend on how $\rho_{13,n}$ and $\rho_{23,n}$ approach zero \emph{individually}. We verify it with simulations shown in \Cref{fig:asymp-WW}. We simulate under $n = 10,000$ for $5,000$ replicates. We set $\rho_{13,n} = r n^{-a}$ and $\rho_{23,n} = t n^{-(1/2-a)}$ such that $\rho_{13,n}\rho_{23,n} = \delta n^{-1/2}$ for $\delta = rt$ under different values of $a$. 
\end{remark}

\begin{figure}[!ht]
\includegraphics[width=0.9\textwidth]{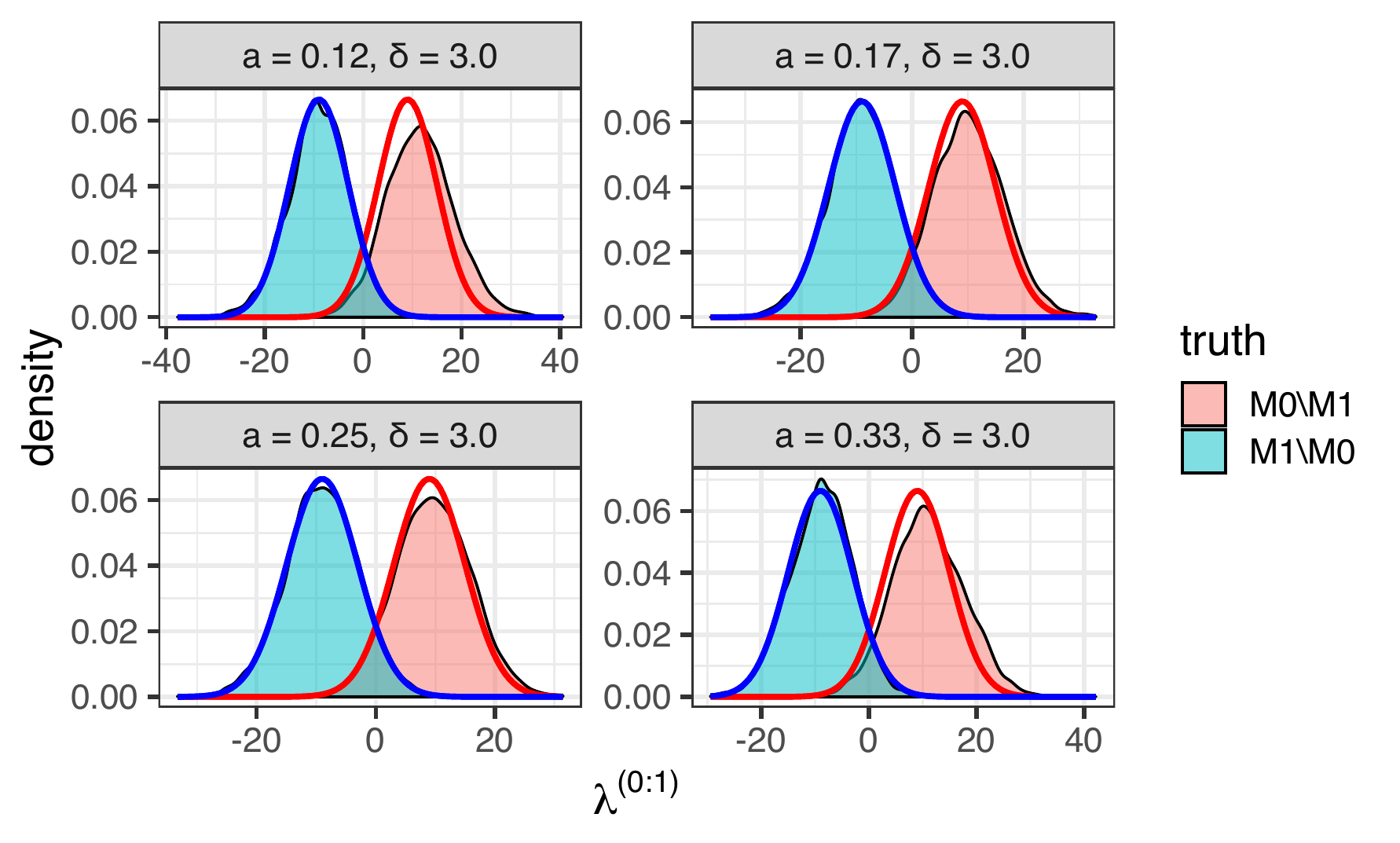}
\caption{Simulated distribution of log-likelihood ratio under $\rho_{13,n} = r n^{-a}$ and $\rho_{23,n} = t n^{-(1/2-a)}$ such that $\rho_{13,n}\rho_{23,n} = \delta n^{-1/2}$ for $\delta = rt$. Red and blue solid curves are theoretical distributions.}
\label{fig:asymp-WW}
\end{figure}
 \subsection{Limit experiments} \label{sec:limit-exp}
We establish the equivalence of testing the two models local asymptotics to that of a limit experiment, which sheds light on the form of the asymptotic distribution. As we will see, the limit experiments are Gaussian location experiments and the problem is asymptotically equivalent to testing the location between two lines from a single normal observation. Further by weak convergence, $\lambda_n^{(0:1)}$ is asymptotically distributed as the likelihood ratio statistic arising from the limit experiment. The reader is referred to \citet[Chapter 7, 9 and 16]{van2000asymptotic} for more background. 

\subsubsection{The weak-strong regime}
We characterize the limit experiment in the weak-strong regime. 

\begin{proposition}
The family of distributions $\{P_{\bm{\Sigma}^{\ast} + \G h / \sqrt{n}}: h \in \mathbb{R}^2\}$ is locally asymptotically normal, where 
\begin{equation*}
\Sigma^{\ast} = \begin{pmatrix}
\sigma_{11} & 0 & 0 \\
0 & \sigma_{22} & \sigma_{23}  \\
0  & \sigma_{23} & \sigma_{33}
\end{pmatrix}, \quad h = (\mathsf{h}_1, \mathsf{h}_2)^{\T},
\end{equation*}
and $\G: \mathbb{R}^2 \rightarrow \mathbb{R}^{3 \times 3}$ is a linear operator 
\begin{equation*}
\G h := \begin{pmatrix}
0 & \mathsf{h}_1 &  \mathsf{h}_2 \\
\mathsf{h}_1  & 0 & 0  \\
\mathsf{h}_2   & 0 & 0
\end{pmatrix}.
\end{equation*}
\end{proposition}
\begin{proof}
The Gaussian model $P_{\Sigma}$ is differentiable in quadratic mean at $\Sigma^{\ast}$. The result follows from \citet[7.14 and 7.15]{van2000asymptotic}.
\end{proof}

The limit experiment of a LAN (locally asymptotically normal) family is a normal location experiment.

\begin{proposition} The sequence of experiments indexed by the local parameter $h$ converges to the following normal location experiment 
\begin{equation}
\left( P_{\Sigma^{\ast} + \G h / \sqrt{n}} \right)_{h \in \mathbb{R}^2} \weaklyconvto \left( \N(h, I_{\Sigma^{\ast}}^{-1}) \right )_{h \in \mathbb{R}^{2}}, \label{eqs:LAN}
\end{equation}
where 
\begin{equation*}
I_{\Sigma^{\ast}}^{-1} = \sigma_{11} \left(
\begin{array}{ccc}
 \sigma _{22} &  \rho \sqrt{\sigma_{22} \sigma_{33}} \\
 \rho \sqrt{\sigma_{22} \sigma_{33}} & \sigma _{33} \\
\end{array}
\right).
\end{equation*}
\end{proposition}
\begin{proof}
$\{P_{\bm{\Sigma}^{\ast} + \G h / \sqrt{n}}: h \in \mathbb{R}^2\}$ is LAN with non-singular Fisher information $I_{\Sigma}^{\ast}$, which is the conditional information matrix of $(\sigma_{12}, \sigma_{13})$ under $P_{\Sigma}$ given $(\sigma_{11}, \sigma_{22}, \sigma_{33}, \sigma_{23})$, corresponding to $(\mathsf{h}_1, \mathsf{h}_2)$. The result then follows from \citet[Corollary 9.5]{van2000asymptotic}.
\end{proof}

The local sequences \cref{eqs:sigma-0-ws} and \cref{eqs:sigma-1-ws} can be identified as $\Sigma^{\ast} + \G h / \sqrt{n}$ with $h$ taking value of
\begin{equation}
h_0 = (0, \gamma \sqrt{\sigma_{11} \sigma_{33}})^{\T}, \quad h_1 = (\gamma \rho \sqrt{\sigma_{11} \sigma_{22}}, \gamma \sqrt{\sigma_{11} \sigma_{33}})^{\T}
\end{equation}
respectively. Models $\model_0$ and $\model_1$ correspond to the set of $h_0$ and $h_1$ respectively as $\gamma$ varies in $\mathbb{R}$. That is, $\model_0$ and $\model_1$ are represented by local parameter spaces
\begin{equation}
H_0 = \{0\} \times \mathbb{R}, \quad H_1 = \{(\gamma \rho \sqrt{\sigma_{11} \sigma_{22}}, \gamma \sqrt{\sigma_{11} \sigma_{33}})^{\T}: \gamma \in \mathbb{R} \},
\end{equation}
which consist of all limits of $\sqrt{n} \G^{-1} (\Sigma_n^{(i)} - \Sigma^{\ast})$ for $i=0,1$ (see \citet[Chapter 7.4]{van2000asymptotic}). Note $H_0$ and $H_1$ are lines in $\mathbb{R}^2$ (affine) and they correspond to tangent cones from $\model_0$ and $\model_1$ at $\Sigma^{\ast}$ under Chernoff regularity; see also \citet{drton2009likelihood} and \citet{geyer1994asymptotics}.

\begin{proposition} \label{prop:limit-exp}
Suppose $I_{\Sigma^{\ast}}^{-1} = L L^{\T}$. For $i=0,1$, under $P^{n}_{\Sigma^{\ast} + \G h / \sqrt{n}}$ for $h = h_i$, it holds that $(-1)^{i} \lambda_n^{(0:1)}$ is asymptotically distributed as the likelihood ratio statistic of testing
\begin{equation}
\mu \in L^{-1} H_i \quad \text{versus} \quad \mu \in L^{-1} (H_{1-i} - h_i)
\end{equation}
from a single observation $Z \sim \N(\mu = \bm{0}, I_2)$. 
\end{proposition}
\begin{proof}
Under $P^{n}_{\Sigma^{\ast} + \G h / \sqrt{n}}$, by \citet[Theorem 16.7]{van2000asymptotic} $\lambda_n^{(0:1)}$ is asymptotically distributed as the log-likelihood ratio statistic for testing $H_0$ and $H_1$ based on a single sample from $\N(h, I_{\Sigma^\ast}^{-1})$. Note that the theorem still applies to our case even though $H_0$ and $H_1$ are non-nested, as its proof does not require the two models to be nested. 
That is, given $X \sim \N(m=\bm{0}, I_{\Sigma^{\ast}}^{-1})$, we have
\begin{equation}
\begin{split}
\lambda_n^{(0:1)} &\distconvto \|I_{{\Sigma}^{\ast}}^{1/2} ({X} + h) - {I}_{{\Sigma}^{\ast}}^{1/2} H_0 \|^2 - \|{I}_{{\Sigma}^{\ast}}^{1/2} ({X} + h) - {I}_{{\Sigma}^{\ast}}^{1/2} H_1 \|^2 \\
&=_{d} \|{I}_{{\Sigma}^{\ast}}^{1/2} {X} - {I}_{{\Sigma}^{\ast}}^{1/2} (H_0 - h )\|^2 - \|{I}_{{\Sigma}^{\ast}}^{1/2} {X} - {I}_{{\Sigma}^{\ast}}^{1/2} (H_1 - h)\|^2,
\end{split}
\end{equation}
which is equivalent to testing $m \in H_0 - h$ versus $m \in H_1 - h$ from $X$.  Given $I_{\Sigma^{\ast}}^{-1} = L L^{\T}$, by rewriting $X =_{d} L Z$ for $Z \sim \N(\mu = \bm{0}, I_2)$, the testing problem is mapped to that from $Z$ by $L^{-1}$. Hence, this is further equivalent to testing 
\begin{equation*}
\mu \in L^{-1} (H_0 - h) \quad \text{versus} \quad \mu \in L^{-1} (H_1 - h)
\end{equation*}
from $Z$. Note that $H_i - h_i = H_i$ since $H_i$ is affine. 
\end{proof}

Now we derive limit experiments based on \cref{prop:limit-exp}. The Cholesky decomposition gives
\begin{equation*}
\begin{split}
L &= \sqrt{\sigma_{11}} \left(
\begin{array}{cc}
\sqrt{\sigma_{22}} & 0 \\
\rho \sqrt{\sigma_{33}} & \sqrt{(1-\rho^2) \sigma_{33}}
\end{array}
\right), \\
L^{-1} &= \frac{1}{\sqrt{\sigma_{11}}} \left(
\begin{array}{cc}
1/\sqrt{\sigma_{22}} & 0 \\
-\rho / \sqrt{(1-\rho^2) \sigma_{22}} & 1/\sqrt{(1-\rho^2) \sigma_{33}}
\end{array}
\right).
\end{split}
\end{equation*}
We have, when $h = h_0$
\begin{equation}
L^{-1} H_0 = \{0\} \times \mathbb{R}, \quad L^{-1} (H_1 - h) = \left \{\begin{pmatrix} 0\\ \frac{-\gamma}{\sqrt{1 - \rho^2}} \end{pmatrix} + u \begin{pmatrix} \rho \\ \sqrt{1 - \rho^2}
\end{pmatrix}: u \in \mathbb{R} \right \}, \label{eqs:lim-exp-ws-0}
\end{equation}
and when $ h = h_1$
\begin{equation}
L^{-1} (H_0 - h) = \{-\gamma \rho\} \times \mathbb{R}, \quad L^{-1} H_1 = \left \{ u \begin{pmatrix} \rho \\ \sqrt{1 - \rho^2}  \end{pmatrix}: u \in \mathbb{R} \right  \}. \label{eqs:lim-exp-ws-1}
\end{equation}

They are visualized in \Cref{fig:experiments}. The limit experiments \cref{eqs:lim-exp-ws-0} and \cref{eqs:lim-exp-ws-1} are of the same type as they are both characterized by an \emph{angle} and an \emph{intercept}. The two have the same angle $\theta = \arcsin \rho$ and their intercepts are related by a factor of $1/\sqrt{1-\rho^2}$. 

\begin{figure}[!ht]
\begin{tikzpicture}[scale=.35]
  \tkzDefPoint(0,0){O}
  \tkzDefPoint(3,0){A}
  \tkzDefPoint(4,4){B}
  \tkzDefPoint(4.5,0){Z}
  \fill[inner color=gray, outer color=white] (Z) circle (2);
  \tkzDrawLine[add=.5 and 2, color=blue](O,A) \tkzLabelLine[pos=2.7, above](O,A){$\model_0$}
  \tkzDrawLine[add=.2 and .5, color=blue](O,B) \tkzLabelLine[pos=1, above=2ex](O,B){$\model_1$}
  \tkzMarkAngle[color=blue](A,O,B)
  \tkzLabelAngle[right, pos=1.4, color=blue](A,O,B){$\theta = \arcsin \rho$}
  \tkzDrawPoints(O,Z)
  \draw[color=blue,decoration={brace,mirror,raise=5pt},decorate] (0,0) -- node[below=6pt] {$\gamma/\sqrt{1-\rho^2}$} (Z);
\end{tikzpicture}
\begin{tikzpicture}[scale=.35]
  \tkzDefPoint(0,0){O}
  \tkzDefPoint(3,0){A}
  \tkzDefPoint(4,4){B}
  \tkzDefPoint(4.5,0){Z}
  \fill[inner color=gray, outer color=white] (Z) circle (2);
  \tkzDrawLine[add=.5 and 2, color=blue](O,A) \tkzLabelLine[pos=2.7, above](O,A){$\model_1$}
  \tkzDrawLine[add=.2 and .5, color=blue](O,B) \tkzLabelLine[pos=1, above=2ex](O,B){$\model_0$}
  \tkzMarkAngle[color=blue](A,O,B)
  \tkzLabelAngle[right, pos=1.4, color=blue](A,O,B){$\theta = \arcsin \rho$}
  \tkzDrawPoints(O,Z)
  \draw[color=blue,decoration={brace,mirror,raise=5pt},decorate] (0,0) -- node[below=6pt] {$\gamma$} (Z);
\end{tikzpicture}
\begin{tikzpicture}[scale=.35]
  \tkzDefPoint(0,0){O}
  \tkzDefPoint(0,1){A}
  \tkzDefPoint(4.5,0){C}
  \tkzDefPoint(4.5,1){B}
  \fill[inner color=gray, outer color=white] (O) circle (2);
  \tkzDrawLine[add=3 and 3, color=blue](O,A) \tkzLabelLine[above=9ex](O,A){$\model_i$}
  \tkzDrawLine[add=3 and 3, color=blue](C,B) \tkzLabelLine[above=9ex](C,B){$\model_{1-i}$}
  \tkzDrawLine[add=.5 and .5, style=dashed, color=blue](O,C)
  \tkzDrawPoints(O)
  \draw[color=blue,decoration={brace,mirror,raise=5pt},decorate] (O) -- node[below=6pt] {$\delta$} (C);
\end{tikzpicture}
\caption{Three limit experiments: (1) $\model_0 \setminus \model_1$ in the weak-strong regime, (2) $\model_1 \setminus \model_0$ in the weak-strong regime, and (3) $\model_i \setminus \model_{1-i}$ in the weak-weak regime for $i=0,1$.}
\label{fig:experiments}
\end{figure}
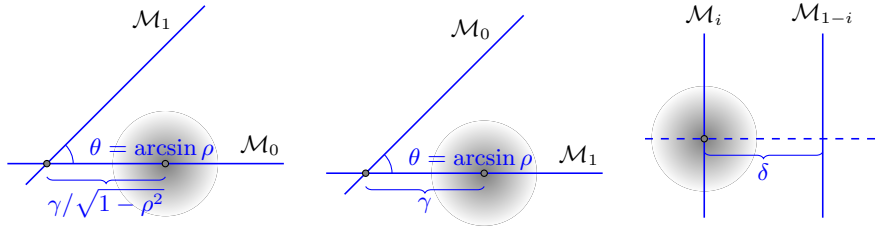

Now we prove the form of the asymptotic distributions in \cref{prop:asymp-weak-strong} from the limit experiment. 

\begin{figure}[!ht]
\begin{tikzpicture}[scale=.5]
  \tkzDefPoint(0,0){O}
  \tkzDefPoint(-3,0){A}
  \tkzDefPoint(3,0){B}
  \tkzDefPoint(0,-3.5){C}
  \tkzDefPoint(2,4){Z}
  \tkzDrawPoints(A,B,O,Z,C) \tkzLabelPoints(B,O) \tkzLabelPoints[above](Z) \tkzLabelPoints[below](A,C)
  \tkzDefPointBy[projection=onto C--B](Z) \tkzGetPoint{F1}
  \tkzDefPointBy[projection=onto C--A](Z) \tkzGetPoint{F2}
  \tkzDefPointBy[projection=onto A--B](Z) \tkzGetPoint{H}
  \tkzDrawLine[add= 1 and 1, color=black, style=dashed](A,B)
  \tkzDrawLine[add= 0 and 2, color=black, style=dashed](C,O)
  \tkzLabelSegment[right, blue](C,B){$\gamma$}
  \tkzDrawSegment[red](H,B) \tkzLabelSegment[red, above](H,B){$Z_1$}
  \tkzDrawLine[add= 0 and 2, color=blue](C,A) \tkzLabelLine[pos=3, left](C,A){$\model_0$}
  \tkzDrawLine[add= 0 and 2, color=blue](C,B) \tkzLabelLine[pos=3, right](C,B){$\model_1$}
  \tkzDrawSegments[color=red](Z,H)  \tkzLabelSegment[red, right](Z,H){$Z_2$}
  \tkzDrawSegments[color=red](F1,Z F2,Z)
  \tkzLabelSegment[red, above left](F2,Z){$d_2$}
  \tkzLabelSegment[red, above right](F1,Z){$d_1$}
  \tkzMarkRightAngles[color=red](Z,F2,C B,F1,Z O,H,Z)
  \tkzMarkAngle[color=blue](B,C,O)
  \tkzLabelAngle[pos=0.5, right=2ex, color=blue](O,C,B){$\theta/2$}
\end{tikzpicture}
\caption{Derivation of the asymptotic distribution \cref{eqs:asymp-dist-10} from the limit experiment of $\model_1 \setminus \model_0$ under the weak-strong regime (the middle panel of \cref{fig:experiments}).}
\label{fig:geo-limit}
\end{figure}
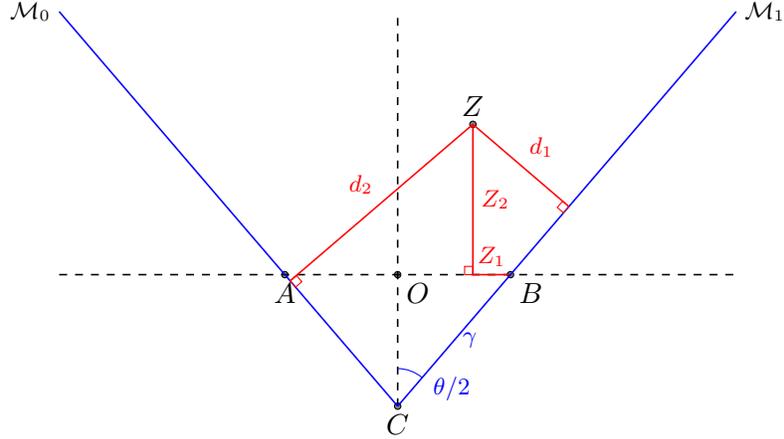

\begin{proof}[Proof of \cref{prop:asymp-weak-strong}]
Since the limit experiments are of the same type, we only derive for local alternatives $\Sigma_n^{(1)} \in \model_1 \setminus \model_0$. We set the coordinate system as in \cref{fig:geo-limit}, where the bisector of angle $\angle B C A = \theta = \arcsin \rho$ is the $y$-axis. The standard Gaussian vector centered at $B$ is represented as $Z=(x, y) = (\gamma \sin (\theta/2) - Z_1, Z_2)$. By the limit experiment, we have $\lambda_n^{(0:1)} \distconvto d_1^2 - d_2^2$. $\model_0$ and $\model_1$ are respectively represented by lines $y = \pm k x + a$ for $k=\cot (\theta/2)$ and $a = -\gamma \cos(\theta/2)$. We have
\begin{equation*}
\begin{split}
d_1^2 - d_2^2 &= \frac{(a + k x - y)^2}{1+k^2} - \frac{(a - k x - y)^2}{1+k^2} \\
&= 2\rho (Z_1 - \gamma \sin (\theta/2))(Z_2 + \gamma \cos(\theta/2)),
\end{split}
\end{equation*}
where we used 
\begin{equation*}
\frac{2k}{1+k^2} = \frac{2 \cot(\theta/2)}{1+\cot^2(\theta/2)} = \sin \theta = \rho.
\end{equation*}
By a change of variables $(Z_1, Z_2) =_{d} ((U+V)/\sqrt{2}, (U-V)/\sqrt{2})$ for another pair of independent standard normals and using the fact
\begin{equation*}
\sqrt{1+\sqrt{1-\rho^2}} - \sqrt{1-\sqrt{1-\rho^2}} = \sqrt{2(1-\rho)},
\end{equation*}
upon simplifying we have
\begin{equation*}
\lambda_{n}^{(0:1)} \distconvto d_1^2 - d_2^2 =_{d} \rho \left [ \left(U + \gamma \sqrt{\frac{1-\rho}{2}} \right)^2 - \left(V + \gamma \sqrt{\frac{1 + \rho}{2}} \right)^2 \right ].
\end{equation*}
\end{proof}

\subsubsection{The weak-weak regime}
The Gaussian limit in \Cref{prop:asymp-ww} shows that the limit experiment of the weak-weak regime is testing the location of a \emph{univariate} normal between two points; see the last panel of \cref{fig:experiments}.

\begin{corollary}
Testing $\model_0$ versus $\model_1$ under $\sqrt{n} \rho_{12,n} \rho_{13,n} \rightarrow \delta$ for $\delta \neq 0$ with $\rho_{12,n}, \rho_{13,n} \rightarrow 0$ is asymptotically equivalent to testing $H_0: \mu = 0$ versus $H_1: \mu = \delta$ from a single observation $Z \sim \N(\mu, 1)$. 
\end{corollary}

 \section{Envelope distributions} \label{sec:envelope}
Though it may at first appear otherwise, the asymptotic distributions as obtained in \cref{prop:asymp-weak-strong} and \cref{prop:asymp-ww} are not directly applicable to forming decision rules. This is due to the \emph{non-uniformity} of the asymptotics. 

Firstly, the asymptotic depends on the \emph{regime}: weak-strong versus weak-weak, namely \emph{where} the local sequence converges to. And the law is \emph{discontinuous} between the two regimes. That is, the law in the weak-strong regime (scaled difference of noncentral chi-squares) does not converge to that of the weak-weak regime (Gaussian) as $\rho \rightarrow 0$. Furthermore, a procedure that firstly estimates the regime and then uses the corresponding distribution to form decision boundary, is susceptible to irregularity issues. Additionally, it is difficult to judge if an edge is weak based on whether its confidence interval contains zero without further assumptions, as illustrated by the following example.

\begin{example}
Suppose $X_i \iid \N(\gamma/\sqrt{n}, \sigma^2)$ for $i=1,\cdots,n$. The usual $(1-\alpha)$-level confidence interval for the mean of $X$ is $\bar{X}_n \pm z_{\alpha/2} \hat{\sigma}_n / \sqrt{n}$. The probability that it contains zero is 
\begin{equation*}
\begin{split}
\Pr\left(0 \in (\bar{X}_n \pm z_{\alpha/2} \hat{\sigma}_n / \sqrt{n}) \right) &= \Pr\left(\sqrt{n} \bar{X}_n / \hat{\sigma}_n \in (\pm z_{\alpha/2}) \right) \\
& \rightarrow \Pr(Z + \gamma \in (\pm z_{\alpha/2})) < 1 - \alpha
\end{split}
\end{equation*}
for $\gamma \neq 0$ and $Z \sim \N(0,1)$. A large enough $\gamma$ can be chosen to make this probability arbitrarily small. 
\end{example}

Secondly, given the regime, the distribution depends on the value of a \emph{local parameter} ($\gamma$ for strong-weak and $\delta$ for weak-weak), which determines \emph{how} the local sequence converges. Due to the $\sqrt{n}$ factor, the standard error for its estimator does not vanish and in general the local parameter cannot be consistently estimated. The reader is referred to \citet{berger1994p,andrews2001testing} for discussions in the literature on the treatment of asymptotic distributions involving nuisance parameters that are not point-identified. Here we take a different approach, presented as follows. 

The non-uniformity of asymptotic distributions motivates us to seek a procedure that \emph{circumvents} the inference on the regime and the local parameter. In this section, we study the ``extremal'' distributions arising from the asymptotic distributions as the local parameter varies in $\mathbb{R}$. 

\begin{definition} \label{def:envelope}
Given a family of distribution functions $\{F_{h}: h \in \mathcal{H}\}$ on $\mathbb{R}$, define
\begin{equation*}
\bar{F}^{\ast}(x) = \sup_{h \in \mathcal{H}} F_h(x),
\end{equation*}
and
\begin{equation}
\bar{F}(x) = \begin{cases} \bar{F}^{\ast}(x), \quad &\text{$\bar{F}^{\ast}$ is continuous at $x$} \\
\lim_{y \rightarrow x^{+}} \bar{F}^{\ast}(y), \quad &\text{$\bar{F}^{\ast}$ is discontinuous at $x$}
\end{cases}.
\end{equation}
We call $\bar{F}$ the envelope distribution of $\{F_{h}: h \in \mathcal{H}\}$ if $\bar{F}$ is a valid distribution function.
\end{definition}

\begin{lemma}
$\bar{F}^{\ast}(x)$ is left-continuous if every $F_h(x)$ for $h \in \mathcal{H}$ is continuous.
\end{lemma}
\begin{proof}
Fix any $x$ and $\delta > 0$, for $\epsilon > 0$ we have $|\bar{F}^{\ast}(x) - \bar{F}^{\ast}(x - \epsilon)| = \sup_{h} F_h(x) - \sup_{h} F_h(x - \epsilon)$. By definition of supremum, there exists $h' \in \mathcal{H}$ such that $F_{h'}(x) \geq \sup_{h} F_h(x) - \delta / 2$. Hence, $|\bar{F}^{\ast}(x) - \bar{F}^{\ast}(x - \epsilon)| \leq \delta/2 + F_{h'}(x) -  F_{h'}(x - \epsilon) $. By continuity of $F_{h'}$, choosing $\epsilon > 0$ such that $F_{h'}(x) -  F_{h'}(x - \epsilon) \leq \delta / 2$ shows that $\bar{F}^{\ast}(x)$ is left-continuous. 
\end{proof}

\begin{lemma} \label{lem:envelope-valid}
If $\bar{F}^{\ast}(x) \rightarrow 0$ as $x \rightarrow -\infty$, then $\bar{F}(x)$ is a valid distribution function.
\end{lemma}
\begin{proof}
Given any $x \leq x'$, $\sup_{h} F_h(x) \leq \sup_{h} F_h(x')$ by monotonicity of every $F_h$. Since $\bar{F}^{\ast}$ is non-decreasing, by \citet[Theorem 3.23]{Folland1999}, the set of points at which $\bar{F}^{\ast}$ is discontinuous is countable. By redefining the function value at these points to be their right limits, $\bar{F}$ is right continuous. Also, $\bar{F}(x) \geq \bar{F}^{\ast}(x) \rightarrow 1$ as $x \rightarrow +\infty$ since every $F_h(x) \rightarrow 1$. Finally, as $x \rightarrow -\infty$ if $\bar{F}^{\ast}(x) \rightarrow 0$ , then $\bar{F}(x) \rightarrow 0$. $\bar{F}$ is a distribution function.
\end{proof}

\subsection{The weak-weak regime}
\begin{proposition} \label{prop:envelope-ww}
Let $G_{\delta} = \{\N(\delta^2, (2\delta)^2): \delta \in \mathbb{R}\}$ be the asymptotic distributions for the weak-weak regime under $\model_0 \setminus \model_1$. The envelope of $\{G_{\delta}\}$ is an equal-probability mixture of $(-\chi_1^2)$ and a point mass at zero, namely
\begin{equation}
\bar{G}(x) = \frac{1}{2} \left(1-F_{\chi_1^2}(-x)\right) \I_{x < 0} + \frac{1}{2} \I_{x \geq 0}
\end{equation}
The corresponding envelope under $\model_1 \setminus \model_0$ is distributed as its negation.
\end{proposition}
\begin{proof}
It suffices to consider $\delta \geq 0$. Given any $x < 0$, 
\begin{equation*}
\sup_{\delta} \Pr(\delta^2 + 2 \delta Z \leq x) = \sup_{\delta >0} \Phi \left(\frac{x - \delta^2}{2 \delta} \right) = \sup_{\delta >0} \Phi \left ( -\left [ \frac{-x}{2 \delta} + \frac{\delta}{2} \right ] \right ) = \Phi(-\sqrt{-x}),
\end{equation*}
where $\delta^{\ast} = \sqrt{-x}$ is the maximizer; Given any $x \geq 0$, $\delta = 0$ maximizes the probability to one. Hence, the envelope CDF is
\begin{equation*}
\bar{G}(x) = \begin{cases}
\Phi(-\sqrt{-x}), \quad x < 0 \\
1, \quad x \geq 0
\end{cases},
\end{equation*}
from which it follows that
\begin{equation*}
\bar{g}(x) = \bar{G}'(x) = \frac{1}{2} f_{\chi_1^2}(-x) \I_{x < 0} + \frac{1}{2} \delta_{0}(x).
\end{equation*} 
The envelope for $\model_1 \setminus \model_0$ follows from symmetry. 
\end{proof}

Since when $\model_0$ is true, the region for rejecting $\model_0$ should take the form $(-\infty, r)$ for some $r < 0$, only the negative part of $\bar{G}$ is relevant for decision making. It follows from \cref{prop:envelope-ww} that the negative part of $\bar{G}$ is distributed as $\chi_1^2$. The formation of the envelope is visualized in \cref{fig:chi-sq-envelope}, which aligns with the behavior observed in \cref{fig:row-gaussians}, where as $\delta$ grows, the quantiles for $\alpha=0.05$ first moves outward for $\delta \in (0.5, 1.64)$ and then moves inward for $\delta \in (1.64, \infty)$.

\begin{figure}[!ht]
\includegraphics[width=0.8\textwidth]{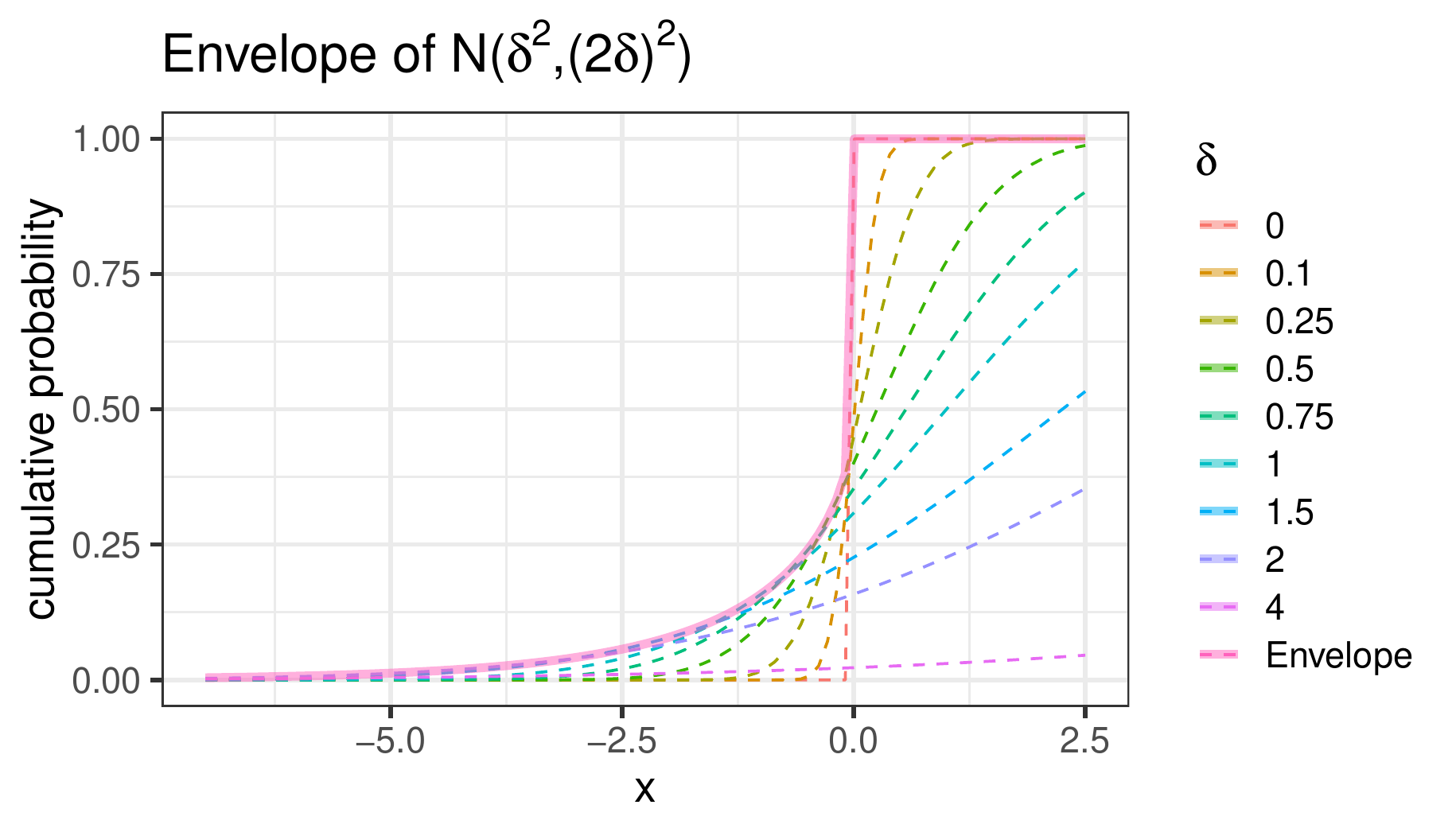}
\caption{The envelope CDF $\bar{G}$ for the weak-weak regime.}
\label{fig:chi-sq-envelope}
\end{figure}

\subsection{The weak-strong regime}
Now we study the envelope distributions under the weak-strong regime. We first observe that the envelope distributions, if they exist, must be symmetric for \cref{eqs:asymp-dist-01} and \cref{eqs:asymp-dist-10}, in the sense that they are distributed as the negation of each other. The symmetry holds because the two local parameters are related by a factor of $1 / \sqrt{1-\rho^2}$ (see \cref{fig:experiments}), and hence the suprema are taken over the same set of laws up to a difference in the sign. Fix $\rho$, let $\{F_{\rho, \gamma}: \gamma \in \mathbb{R}\}$ be the family of asymptotic distributions in the weak-strong regime under $\model_0 \setminus \model_1$ as given in \cref{eqs:asymp-dist-01}. Let $\bar{F}_{\rho}$ be its envelope distribution function. 
\begin{proposition}
$\bar{F}_{\rho}$ is a valid distribution function for $|\rho| \in (0, 1]$.
\end{proposition}
\begin{proof}
Note since $\bar{F}_{\rho} = \bar{F}_{-\rho}$, it suffices to consider $\rho \in (0,1]$. 
First consider $\varphi_{\rho, \gamma}(x)$, the density function for $X^2 - Y^2$ with $X \sim \N\left(\mu_1 = \gamma \sqrt{\frac{1-\rho}{2}}, 1 \right)$ and $Y \sim \N \left(\mu_2=\gamma \sqrt{\frac{1+\rho}{2}}, 1 \right)$ for $\gamma \in \mathbb{R}$ and $\rho \in (0, 1]$.  Since $p(X^2 - Y^2 = v^2, Y^2=t) = p(Y^2=t) p(X^2 = t + v^2 )$, the density $\varphi_{\rho, \gamma}$ has the following integral representation from marginalization
\begin{equation*}
\begin{split}
\varphi_{\rho, \gamma}(v^2) &= \int_{0}^{\infty} \chi_1^2(t; \mu_2^2) \chi_1^2(v^2 + t; \mu_1^2) \dd t\\
&= \frac{1}{2\pi} e^{-v^2 / 2 -\gamma^2 / 2} \int_{0}^{\infty} \frac{\exp(-t) \cosh(\gamma \sqrt{\frac{1+\rho}{2}} \sqrt{t}) \cosh(\gamma \sqrt{\frac{1 - \rho}{2}} \sqrt{t + v^2})}{\sqrt{t(t + v^2)}} \dd t.
\end{split}
\end{equation*}
Recall that $\lesssim$ allows for a positive multiplicative constant. Using $\cosh(x) < \exp(x)$ for $x>0$, we have
\begin{equation*}
\begin{split}
\varphi_{\rho,\gamma}(v^2) & \lesssim e^{-v^2/2 - \gamma^2/2} \int_{0}^{\infty} \frac{e^{-t} \cosh(\gamma \sqrt{\frac{1+\rho}{2}} \sqrt{t}) \cosh(\gamma \sqrt{\frac{1-\rho}{2}} \sqrt{t+v^2})}{\sqrt{t(t+v^2)}} \dd t \\
&< e^{-v^2/2} \int_{0}^{\infty} \frac{\exp\left(-t - \gamma^2/2 + \gamma \sqrt{\frac{1+\rho}{2}} \sqrt{t} + \gamma \sqrt{\frac{1-\rho}{2}} \sqrt{t+v^2} \right)}{\sqrt{t(t+v^2)}} \dd t.
\end{split}
\end{equation*}
We note that
\begin{equation*}
- \gamma^2/2 + \gamma \sqrt{\frac{1+\rho}{2}} \sqrt{t} + \gamma \sqrt{\frac{1-\rho}{2}} \sqrt{t+v^2} \leq \frac{1}{2} \left( \sqrt{\frac{1+\rho}{2}} \sqrt{t}  +  \sqrt{\frac{1-\rho}{2}} \sqrt{t+v^2} \right)^2
\end{equation*}
by completing the square in $\gamma$. It then follows that
\begin{equation*} \label{eqs:density-bound}
\begin{split}
\varphi_{\rho, \gamma}(v^2) &< e^{-v^2/2} \int_{0}^{\infty} \frac{\exp\left [ -t + \frac{1}{2} \left(t + \frac{1-\rho}{2} v^2 + \sqrt{1-\rho^2} \sqrt{t(t+v^2)} \right)\right]}{\sqrt{t(t+v^2)}} \dd t \\
&= e^{-\frac{1+\rho}{4} v^2} \int_{0}^{\infty} \frac{\exp \left ( -\frac{t}{2} + \frac{\sqrt{1-\rho^2}}{2} \sqrt{t(t+v^2)}\right)}{\sqrt{t(t+v^2)}} \dd t \\
&\leq e^{-\frac{1+\rho}{4} v^2} \int_{0}^{\infty} \frac{\exp \left ( -\frac{t}{2} + \frac{\sqrt{1-\rho^2}}{2} (t + v^2/2) \right)}{\sqrt{t(t+v^2)}} \dd t \\
&= e^{-\frac{1+\rho - \sqrt{1 - \rho^2}}{4} v^2} \int_{0}^{\infty} \frac{\exp \left ( -\frac{1 - \sqrt{1-\rho^2}}{2}t \right)}{\sqrt{t(t+v^2)}} \dd t \\
&= e^{-\frac{\rho}{4} v^2} K_0\left(\frac{1-\sqrt{1-\rho^2}}{4} v^2\right),
\end{split}
\end{equation*}
where we used $2 \sqrt{t(t+v^2)} \leq 2t + v^2$ in the third line. $K_{\nu}(\cdot)$ is the modified Bessel function of the second kind, and has the following asymptotic expansion for $z > 0$ \citep[Page 378]{abramowitz1972handbook}
\begin{equation*}
K_{\nu}(z) = \sqrt{\frac{\pi}{2z}} e^{-z} (1 + \frac{4 \nu^2 - 1}{8z} + O(z^{-2})).
\end{equation*}
Hence for large $v^2$, we have
\begin{equation*}
\varphi_{\rho, \gamma}(v^2) \lesssim \frac{\exp\left( - \frac{1 + \rho - \sqrt{1 - \rho^2}}{4} v^2 \right)}{(1 - \sqrt{1 - \rho^2}) v^2}.
\end{equation*}
Recall that $\{F_{\rho, \gamma}: \gamma \in \mathbb{R}\}$ is the family of distributions for the RHS of \cref{eqs:asymp-dist-01}. With $\gamma' = \sqrt{1 - \rho^2} \gamma$, 
\begin{equation*}
\rho \left [ \left(Z_1 + \gamma \sqrt{\frac{1+\rho}{2}} \right)^2 - \left(Z_2 + \gamma \sqrt{\frac{1 - \rho}{2}} \right)^2 \right ] \sim F_{\rho, \gamma'}.
\end{equation*}
It follows that the density function
\begin{equation}
f_{\rho, \gamma'}(-v^2) = \varphi_{\rho, \gamma}(v^2 / \rho) \lesssim \frac{\rho \exp\left( - \frac{1 + \rho - \sqrt{1 - \rho^2}}{4 \rho} v^2 \right)}{(1 - \sqrt{1 - \rho^2}) v^2},
\end{equation}
where the exponent $\frac{1 + \rho - \sqrt{1 - \rho^2}}{4 \rho} \in (1/4, 1/2]$. By \cref{def:envelope}, we have
\begin{equation*}
\begin{split}
\bar{F}^{\ast}_{\rho}(-v^2) &= \sup_{\gamma' \in \mathbb{R}} F_{\rho, \gamma'}(-v^2) \\
& = \sup_{\gamma' \in \mathbb{R}} \int_{v^2}^{\infty} f_{\rho, \gamma'} (-u) \dd u \lesssim \int_{v^2}^{\infty} \frac{\rho \exp\left( - \frac{1 + \rho - \sqrt{1 - \rho^2}}{4 \rho} u \right)}{(1 - \sqrt{1 - \rho^2}) u} \dd u < \infty,
\end{split}
\end{equation*}
and hence $\bar{F}^{\ast}_{\rho}(-v^2) \rightarrow 0$ as $v \rightarrow \infty$. By \cref{lem:envelope-valid}, $\bar{F}_{\rho}$ is a distribution function for every $\rho \in (0, 1]$.
\end{proof}

The following proposition shows that $F_{\rho, \gamma=0}$ constitutes the envelope for the positive part of $\bar{F}_\rho$. 

\begin{proposition} \label{prop:positive-part-envelope}
The positive part of $\bar{F}_{\rho}$ for $|\rho| \in (0, 1]$ is distributed as the positive part of $\rho(Z_1^2 - Z_2^2)$ for $Z_1, Z_2 \iid \N(0,1)$. 
\end{proposition}

\begin{proof}
Fix $\rho \in (0,1]$ and $v^2 \geq 0$, with $\gamma' = \gamma \sqrt{1-\rho^2}$ it follows from \cref{prop:asymp-weak-strong} that 
\begin{equation}
1 - F_{\rho, \gamma'}(v^2) = \Pr \left ( (Z_1 + \mu_1)^2 - (Z_2 + \mu_2)^2 \geq v^2 / \rho \right ),\label{eqs:gaussian-hyperbola}
\end{equation}
where $\mu_1 = \gamma \sqrt{(1+\rho)/2}$, $\mu_2 = \gamma \sqrt{(1-\rho)/2}$. Since $F_{\rho, \gamma'}$ is symmetric in $\gamma$, we show $\gamma = 0$ maximizes $F_{\rho, \gamma'}(v^2)$ by showing that the probability on the RHS of \cref{eqs:gaussian-hyperbola} increases in $\gamma \in (0, \infty)$. The probability can be interpreted as the standard Gaussian measure of the hyperbolic set $\{(x,y): x^2  - y^2 \geq v^2 / \rho\}$ with the Gaussian centered at $G=(\mu_1, \mu_2) = \gamma(\sqrt{(1+\rho)/2},  \sqrt{(1-\rho)/2})$. This is visualized in \cref{fig:geo-hyperbola}, where $\gamma = \overline{OG}$, $\tan \phi = \sqrt{(1-\rho)/(1+\rho)}$ and the hyperbolic set consists of the area inside the two branches. As $\gamma$ increases from zero, the center moves away from the origin along the $V$ line. Let $U$ be the line perpendicular to $V$. The Gaussian measure has two independent standard normal projections $(U, V)$, which is a rotation of $(Z_1, Z_2)$. Now we show that for every $v>0$, by conditioning on $|V| = v$, the conditional probability of $U$ in the appropriate ``section'' of the hyperbolic set, denoted by probability $q(v)$, increases with $\gamma$. 

Let $[A,B]$ and $[C,D]$ be the line segments that $V =-v$ and $V=v$ intersect the hyperbola respectively. By independence of $U$ and $V$, we have $q(v) = \Pr(U \in [A,B]) + \Pr(U \in [C,D])$. Let $\hat{v}$ and $\bar{v}$ be the distance from $G$ to the tangent to the left and right branch of the hyperbola respectively, parallel to line $U$; see \cref{fig:geo-hyperbola}. There are three cases. (i) When $v \leq \bar{v}$ (the first panel of \cref{fig:geo-hyperbola}), as $\gamma$ increases, both $[A,B]$ and $[C,D]$ become bigger, and thus $q(v)$ increases. (ii) When $\bar{v} < v \leq \hat{v}$, $[A,B]$ is empty but $[C,D]$ becomes bigger, so $q(v)$ increases. (iii) When $v > \hat{v}$, as $\gamma$ increases (the second panel of \cref{fig:geo-hyperbola}), $[C,D]$ increases but $[A,B]$ decreases. Let $[E,F]$ be the segment symmetric to $[A,B]$ about the origin. We observe that, as $\gamma$ increases by an infinitesimal $\Delta \gamma$, the amount that $\Pr(U \in [A,B])$ decreases equals the amount that $\Pr(U \in [E,F])$ increases, which is smaller than the amount that $\Pr(U \in [C,D])$ increases. Hence, $q(v)$ still increases. 

By the monotonicity for every value of $|V|$, we conclude that the total probability on the RHS of \cref{eqs:gaussian-hyperbola} increases in $\gamma$. Hence, $F_{\gamma, \rho}(v^2)$ is maximized at $\gamma = 0$ for every $v$, namely $\bar{F}_{\rho} = F_{\rho, \gamma=0}$. It follows that for $X \sim \bar{F}_{\gamma}$, $(X)_{+} =_{d} \rho(Z_1^2 - Z_2^2)_{+}$ for two independent standard normal variables $Z_1, Z_2$. 
\end{proof}

\begin{corollary}
$\bar{F}_{\rho}(0) \equiv 1/2$.
\end{corollary}
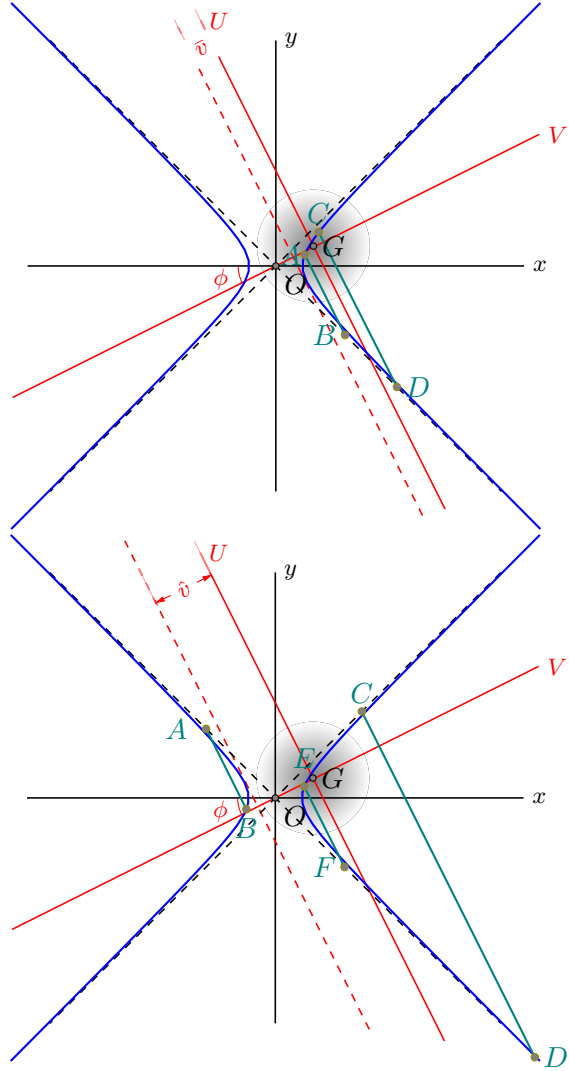
\begin{figure}[!ht]
\centering
\begin{tikzpicture}[scale=.5]
  \tkzDefPoint(0,0){O}
  \tkzDefPoint(0.77,0.3){A}
  \tkzDefPoint(0.7,-0.24){J}
  \tkzDefPoint(-2.2,0){J1}
  \tkzDefPoint(2.2,0){J2}
  \tkzDefPoint(0,2){E}
  \tkzDefPoint(0,-2){F}
  \tkzDefPoint(2,2){C1}
  \tkzDefPoint(-2,-2){C2}
  \tkzDefPoint(1, .5){Z1}
  \tkzDefPoint(-1, -.5){Z2}
  \tkzDefPoint(1, 0.53){G}
  \tkzDefPoint(-2,2){D1}
  \tkzDefPoint(2,-2){D2}
  \tkzDefPointsBy[symmetry=center G](A){M}
  \tkzDefLine[orthogonal=through G](Z1,Z2) \tkzGetPoint{H}
  \tkzDefLine[parallel=through M](G,H) \tkzGetPoint{M'}
  \tkzDefPoint(1.14,0.9){C}
  \tkzInterLL(M,M')(D1,D2) \tkzGetPoint{D}
  \tkzDefLine[parallel=through A](G,H) \tkzGetPoint{A'}
  \tkzInterLL(A,A')(D1,D2) \tkzGetPoint{B}
  \tkzDefLine[parallel=through J](G,H) \tkzGetPoint{J'}
  \tkzDefPointWith[linear,K=3.7](H,G) \tkzGetPoint{S}
  \tkzDefPointsBy[projection=onto J--J'](S){S'}
  \fill[inner color=gray, outer color=white] (G) circle (1.5);
  \tkzDrawLine[add= 1 and 1, color=black, end=$x$](J1,J2)
  \tkzDrawLine[add= 1 and 1, color=black, end=$y$](F,E)
  \tkzDrawLine[add= 3 and 3, color=red, end=$V$](Z2,Z1)
  \tkzDrawLine[add= 2.5 and 3, color=red, end=$U$](H,G)
  \tkzDrawLine[add= 2.2 and 3.3, color=red, end=$$, style=dashed](J',J)
  \dimline[color=red]{(S)}{(S')}{$\bar{v}$}
  \tkzDrawLines[add= 1 and 1, color=black, style=dashed](C1,C2 D1,D2)
  \draw[domain=-7:7, samples=50, color=blue, thick] plot ({sqrt(\x*\x + 0.5)}, {\x});
  \draw[domain=-7:7, samples=50, color=blue, thick] plot ({-sqrt(\x*\x + 0.5)}, {\x});
  \tkzDrawSegment[color=teal, thick](C,D)
  \tkzDrawSegment[color=teal, thick](A,B)
  \tkzDrawPoints[color=olive, size=7](C,D) \tkzLabelPoints[right, color=teal](D) \tkzLabelPoints[above, color=teal](C)
  \tkzDrawPoints[color=olive, size=7](A,B) \tkzLabelPoints[left, color=teal](B) \tkzLabelPoints[left=-0.5ex, color=teal](A)
  \tkzDrawPoints(O,G) \tkzLabelPoints(O) \tkzLabelPoints[right](G)
  \tkzMarkAngle[color=red](J1,O,Z2)
  \tkzLabelAngle[pos=-1.5, color=red](Z2,O,J1){$\phi$}
\end{tikzpicture}
\begin{tikzpicture}[scale=.5]
  \tkzDefPoint(0,0){O}
  \tkzDefPoint(1.8,1.7){C}
  \tkzDefPoint(0.77,0.3){E}
  \tkzDefPoint(-0.7,0.24){J}
  \tkzDefPoint(-2.2,0){J1}
  \tkzDefPoint(2.2,0){J2}
  \tkzDefPoint(0,2){P}
  \tkzDefPoint(0,-2){Q}
  \tkzDefPoint(2,2){C1}
  \tkzDefPoint(-2,-2){C2}
  \tkzDefPoint(1, .5){Z1}
  \tkzDefPoint(-1, -.5){Z2}
  \tkzDefPoint(1, 0.53){G}
  \tkzDefPoint(-2,2){D1}
  \tkzDefPoint(2,-2){D2}
  \tkzDefLine[orthogonal=through G](Z1,Z2) \tkzGetPoint{H}
  \tkzDefLine[parallel=through C](G,H) \tkzGetPoint{C'}
  \tkzInterLL(C,C')(D1,D2) \tkzGetPoint{D}
  \tkzDefLine[parallel=through E](G,H) \tkzGetPoint{E'}
  \tkzInterLL(E,E')(D1,D2) \tkzGetPoint{F}
  \tkzDefLine[parallel=through J](G,H) \tkzGetPoint{J'}
  \tkzDefPointWith[linear,K=3.7](H,G) \tkzGetPoint{S}
  \tkzDefPointsBy[projection=onto J--J'](S){S'}
  \tkzDefPointsBy[symmetry=center O](E,F){B,A}
  \tkzDefPointsBy[symmetry=center G](B){M}
  \tkzDefLine[parallel=through M](G,H) \tkzGetPoint{M'}
  \tkzInterLL(M,M')(D1,D2) \tkzGetPoint{D}
  \tkzInterLL(M,M')(C1,C2) \tkzGetPoint{C}
  \fill[inner color=gray, outer color=white] (G) circle (1.5);
  \tkzDrawLine[add= 1 and 1, color=black, end=$x$](J1,J2)
  \tkzDrawLine[add= 1 and 1, color=black, end=$y$](Q,P)
  \tkzDrawLine[add= 3 and 3, color=red, end=$V$](Z2,Z1)
  \tkzDrawLine[add= 2.5 and 3, color=red, end=$U$](H,G)
  \tkzDrawLine[add= 2.2 and 3.3, color=red, end=$$, style=dashed](J',J)
  \dimline[color=red]{(S)}{(S')}{$\hat{v}$}
  \tkzDrawLines[add= 1 and 1, color=black, style=dashed](C1,C2 D1,D2)
  \draw[domain=-7:7, samples=50, color=blue, thick] plot ({sqrt(\x*\x + 0.5)}, {\x});
  \draw[domain=-7:7, samples=50, color=blue, thick] plot ({-sqrt(\x*\x + 0.5)}, {\x});
  \tkzDrawSegment[color=teal, thick](C,D)
  \tkzDrawPoints[color=olive, size=7](C,D) \tkzLabelPoints[above, color=teal](C) \tkzLabelPoints[right, color=teal](D)
  \tkzDrawSegment[color=teal, thick](E,F)
  \tkzDrawPoints[color=olive, size=7](E,F) \tkzLabelPoints[left, color=teal](F) \tkzLabelPoints[above=1ex, color=teal](E)
  \tkzDrawSegment[color=teal, thick](A,B)
  \tkzDrawPoints[color=olive, size=7](A,B) \tkzLabelPoints[below, color=teal](B) \tkzLabelPoints[left=1ex, color=teal](A)
  \tkzDrawPoints(O,G) \tkzLabelPoints(O) \tkzLabelPoints[right](G)
  \tkzMarkAngle[color=red](J1,O,Z2)
  \tkzLabelAngle[pos=-1.5, color=red](Z2,O,J1){$\phi$}
\end{tikzpicture}

\caption{The CDF $F_{\rho, \gamma}(\cdot)$ at can be interpreted as the probability of a hyperbolic set (inside the two branches of blue curves) as measured by a standard normal centered $|\gamma|$ away from the origin, lying on the line $V$ with slope $\tan \phi = \sqrt{(1-\rho)/(1+\rho)}$. The asymptotes of the hyperbola are $y=\pm x$.}
\label{fig:geo-hyperbola}
\end{figure}

Unfortunately, we do not have an analytic form of the distribution for the negative part of $\bar{F}_{\rho}$, which is the part relevant for decision making, except for $\rho \rightarrow 0$ and $\rho = 1$. 

\begin{proposition}[Bessel envelope] \label{prop:Bessel-K}
$\bar{F}_{\rho=1} =_{d} Z_1^2 - Z_2^2$ for $Z_1, Z_2 \iid \N(0,1)$.
\end{proposition}
\begin{proof}
Under $\rho = 1$, the CDF is
\begin{equation*}
F_{\gamma}(x) = \Pr\left ( (Z_1 + \gamma)^2 -Z_2^2 \leq x \right) = \E_{Z_2} \Pr\left ( (Z_1 + \gamma)^2 \leq x + Z_2^2 \mid Z_2 \right).
\end{equation*}
Since the conditional probability is non-negative, it suffices to show that given any $x \in \mathbb{R}$, $\gamma = 0$ maximizes $\Pr\left((Z_1 + \gamma)^2 \leq x + z_2^2 \mid Z_2 = z_2 \right) = \Pr\left((Z_1 + \gamma)^2 \leq x + z_2^2 \right)$ for all $z_2 \in \mathbb{R}$. When $x + z_2^2 \leq 0$, the conditional probability is zero and $\gamma = 0$ is trivially a maximizer. When $x + z_2^2 > 0$, then $\Pr\left((Z_1 + \gamma)^2 \leq x + z_2^2 \right) = \Phi(\sqrt{x+z_2^2} - \gamma) - \Phi(-\sqrt{x + z_2^2} - \gamma)$. Setting the derivative with respect to $\gamma$ to zero requires $\phi(-\sqrt{x + z_2^2} - \gamma) = \phi(\sqrt{x + z_2^2} - \gamma)$, to which $\gamma = 0$ is the unique solution. Therefore, $\gamma = 0$ is the unique maximizer of $F_{\gamma}(x)$ for all $x$. 
\end{proof}

The distribution in \cref{prop:Bessel-K} is a difference between two independent $\chi_1^2$ variables. The density, as plotted in \cref{fig:bessel}, is
\begin{equation*}
p_{B}(u) = \frac{1}{2 \pi} K_0(|u| / 2), \label{eqs:pdf-Bessel}
\end{equation*}
where $K_{0}$ is a modified Bessel function of the second kind. It is referred to as a $K$-form Bessel distribution in the literature; see \citet[Chapter 4.4]{johnson1995continuous}, \citet{bhattacharyya1942use} and \citet[Page 25]{simon2007probability}.

\begin{figure}[!ht]
\includegraphics[width=0.6\textwidth]{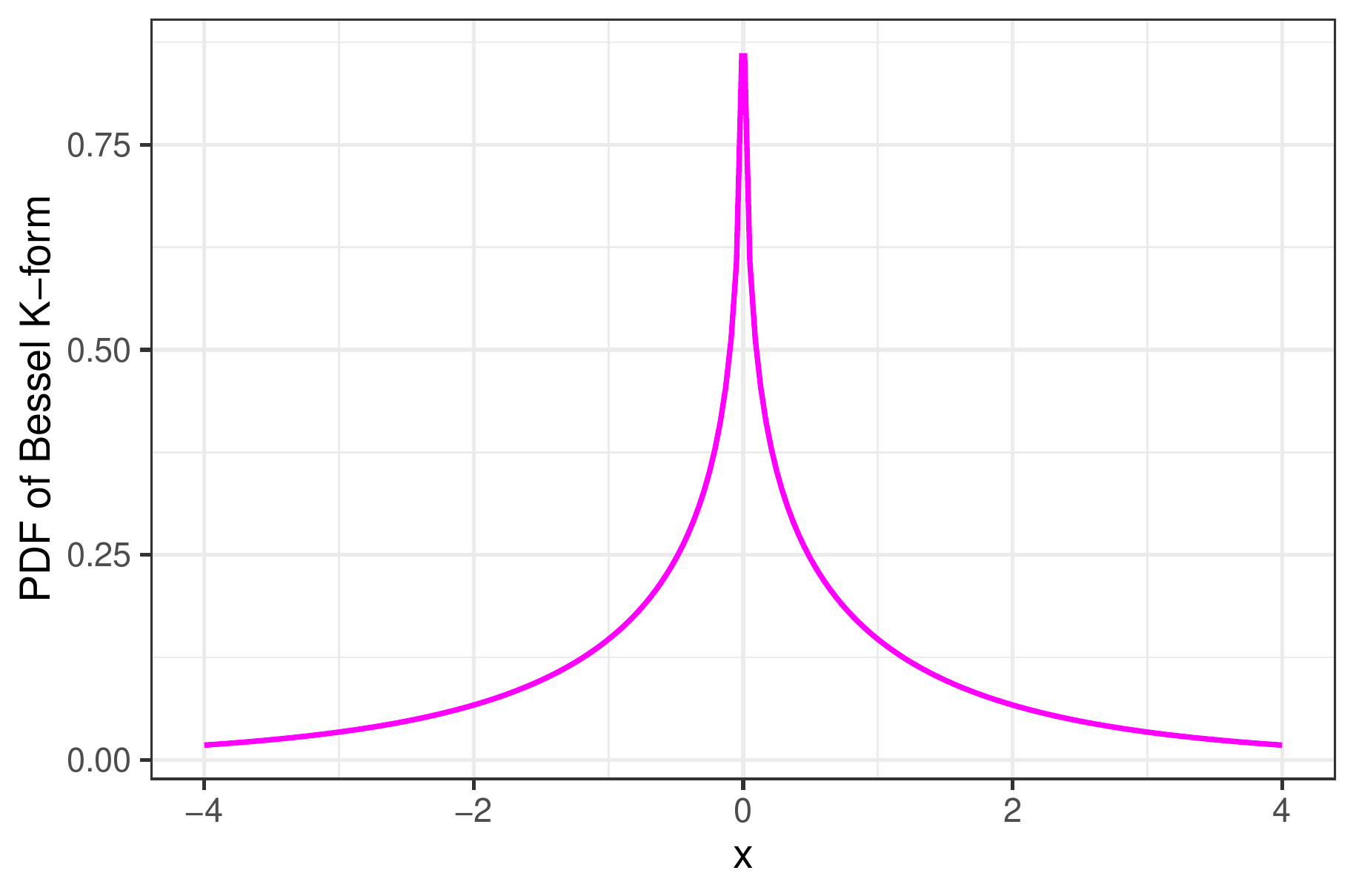}
\caption{The density for $\bar{F}_{\rho=1} =_{d} Z_1^2 - Z_2^2$.}
\label{fig:bessel}
\end{figure}

\begin{proposition}[Continuity of envelope] \label{prop:envelope-continuity}
 $\bar{F}_{\rho} \distconvto \bar{G}$ as $\rho \rightarrow 0$, where $\bar{G}$ is the envelope distribution for the weak-weak regime given in \cref{prop:envelope-ww}.
\end{proposition}
\begin{proof}
Firstly, we note that $\bar{F}_{\rho}(0) = \bar{G}(0) = 1/2$ and by \cref{prop:positive-part-envelope} the non-negative part of $\bar{F}_{\rho}$ also converges to that of $\bar{G}$ as $\rho \rightarrow 0$, namely a point mass at zero.
It remains to be shown that the negative part of $\bar{F}_{\rho}$ converges in law to the negative part of $\bar{G}$.
It suffices to show for any $x \leq 0$
\begin{equation*}
\sup_{\gamma} \Pr\left ( \rho \left[ \left(Z_1 + \gamma \sqrt{\frac{1 + \rho}{2}} \right )^2 - \left(Z_2 + \gamma \sqrt{\frac{1 - \rho}{2}} \right )^2 \right] \leq x \right) \rightarrow \Pr(-Z^2 \leq x) / 2
\end{equation*}
as $\rho \rightarrow 0$. Given $\rho >0$, the maximized probability can be rewritten as
\begin{equation*}
\begin{split}
&\quad \sup_{\gamma} \Pr\left ( \rho \left[ \left(Z_1 + \gamma \sqrt{\frac{1 + \rho}{2}} \right )^2 - \left(Z_2 + \gamma \sqrt{\frac{1 - \rho}{2}} \right )^2 \right] \leq x \right) \\
&= \sup_{\gamma} \Pr\left (  (\gamma \rho)^2  + 2 \gamma \rho \left( \sqrt{\frac{1 + \rho}{2}} Z_1 - \sqrt{\frac{1 - \rho}{2}} Z_2 \right ) \leq x - \rho(Z_1^2 - Z_2^2) \right ) \\
&= \sup_{\delta} \Pr\left (  \delta^2  + 2 \delta \left( \sqrt{\frac{1 + \rho}{2}} Z_1 - \sqrt{\frac{1 - \rho}{2}} Z_2 \right ) + \rho(Z_1^2 - Z_2^2) \leq x  \right ) \\
&= \sup_{\delta} \Pr\left (X_{\rho}(\delta) \leq x \right ),
\end{split}
\end{equation*}
where we define
\begin{equation*}
X_{\rho}(\delta) := \delta^2  + 2 \delta \left( \sqrt{\frac{1 + \rho}{2}} Z_1 - \sqrt{\frac{1 - \rho}{2}} Z_2 \right ) + \rho(Z_1^2 - Z_2^2)
\end{equation*}
for $\rho \in [0,1)$ and $\delta \in \mathbb{R}$. Note that $\sup_{\delta} \Pr(X_{0}(\delta) \leq x ) = \sup_{\delta} \Pr(\delta^2 + 2 \delta Z \leq x) = \Pr(-Z^2 \leq x) / 2$ for $Z \sim \N(0,1)$ by \cref{prop:envelope-ww}. We are left to show $\sup_{\delta} \Pr\left (X_{\rho}(\delta) \leq x \right ) \rightarrow \sup_{\delta} \Pr\left (X_{0}(\delta) \leq x \right )$ as $\rho \rightarrow 0$. Choose $x < M < \infty$ and define 
\begin{equation*}
Y_{\rho}(\delta) := \begin{cases} X_{\rho}(\delta), &\quad X_{\rho}(\delta) \leq M\\
M, &\quad X_{\rho}(\delta) > M
\end{cases}.
\end{equation*}
We observe that 
\begin{equation*}
\begin{split}
&\quad \left | \sup_{\delta} \Pr\left (X_{\rho}(\delta) \leq x \right ) - \sup_{\delta} \Pr\left (X_{0}(\delta) \leq x \right )  \right | \\
&= \left | \sup_{\delta} \Pr\left (Y_{\rho}(\delta) \leq x \right ) - \sup_{\delta} \Pr\left (Y_{0}(\delta) \leq x \right )  \right | \\
&\leq \sup_{\delta}  \left |  \Pr\left (Y_{\rho}(\delta) \leq x \right ) - \Pr\left (Y_{0}(\delta) \leq x \right )  \right | \\
& =  \sup_{\delta}  \left |  \E \left (\I_{Y_{\rho}(\delta) \leq x} - \I_{Y_{0}(\delta) \leq x} \right ) \right | \rightarrow 0,
\end{split}
\end{equation*}
where the last step follows from weak convergence $\{Y_{\rho}(\delta): \delta \in \mathbb{R}\} \weaklyconvto \{Y_{0}(\delta): \delta \in \mathbb{R}\}$ in $\ell^{\infty}(\mathbb{R})$ as $\rho \rightarrow 0$ for a bounded stochastic process; see \citet[Chap. 18]{van2000asymptotic}. 
\end{proof}

Perhaps surprisingly, \cref{prop:envelope-continuity} shows that the asymptotic envelope is \emph{continuous} between the two regimes, which bridges the discontinuity of the asymptotic distributions of $\lambda_n^{(0:1)}$ as presented in \cref{prop:asymp-weak-strong,prop:asymp-ww}. Therefore, taking the envelope resolves the non-uniformity issue in terms of \emph{both} the regime and the local parameter. Now with this we extend the definition of the envelope $\bar{F}_{\rho}$ to $\rho \in [0,1]$ by writing $\bar{F}_{\rho=0} = \bar{G}$. 

\begin{figure}[!ht]
\includegraphics[width=0.9\textwidth]{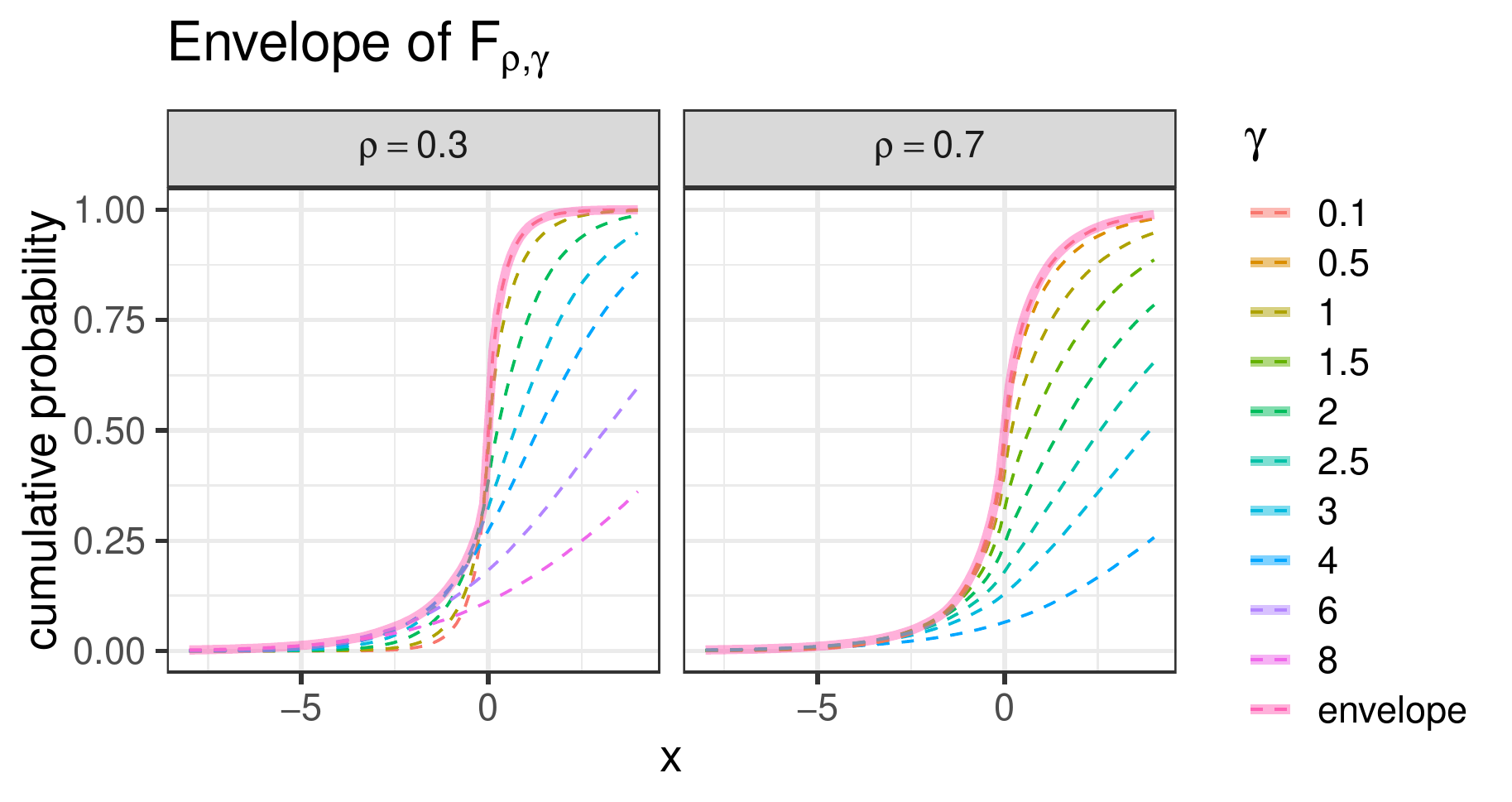}
\caption{The envelope distribution $\bar{F}_{\rho}$ under the strong-weak regime for $\rho = 0.3, 0.7$.}
\label{fig:envelope-strong-weak}
\end{figure}

\Cref{fig:envelope-strong-weak} showcases two envelopes. In the absence of an analytic form for $\rho \in (0, 1)$, the envelopes can be numerically simulated by taking the supremum over a grid of values for $\gamma$. We observe from simulations that there exists $\gamma^{\ast}(x, \rho) \in (0, \infty)$ such that $\pm \gamma^{\ast}$ uniquely maximizes $F_{\gamma, \rho}(x)$. 

Finally, we conclude this section by noting the following \emph{envelope of envelopes}. See \Cref{fig:envelope-of-envelopes} for an illustration. This result will be used in the next section to form simple decision rules based on the Bessel distribution. 

\begin{proposition}[Envelope of envelopes] \label{prop:envelope-of-envelopes}
The negative part of the envelope of $\{\bar{F}_{\rho}: \rho \in [0,1] \}$ is distributed as the negative part of $\bar{F}_{\rho = 1}$ (Bessel).
\end{proposition}

\begin{figure}[!ht]
\includegraphics[width=0.8\textwidth]{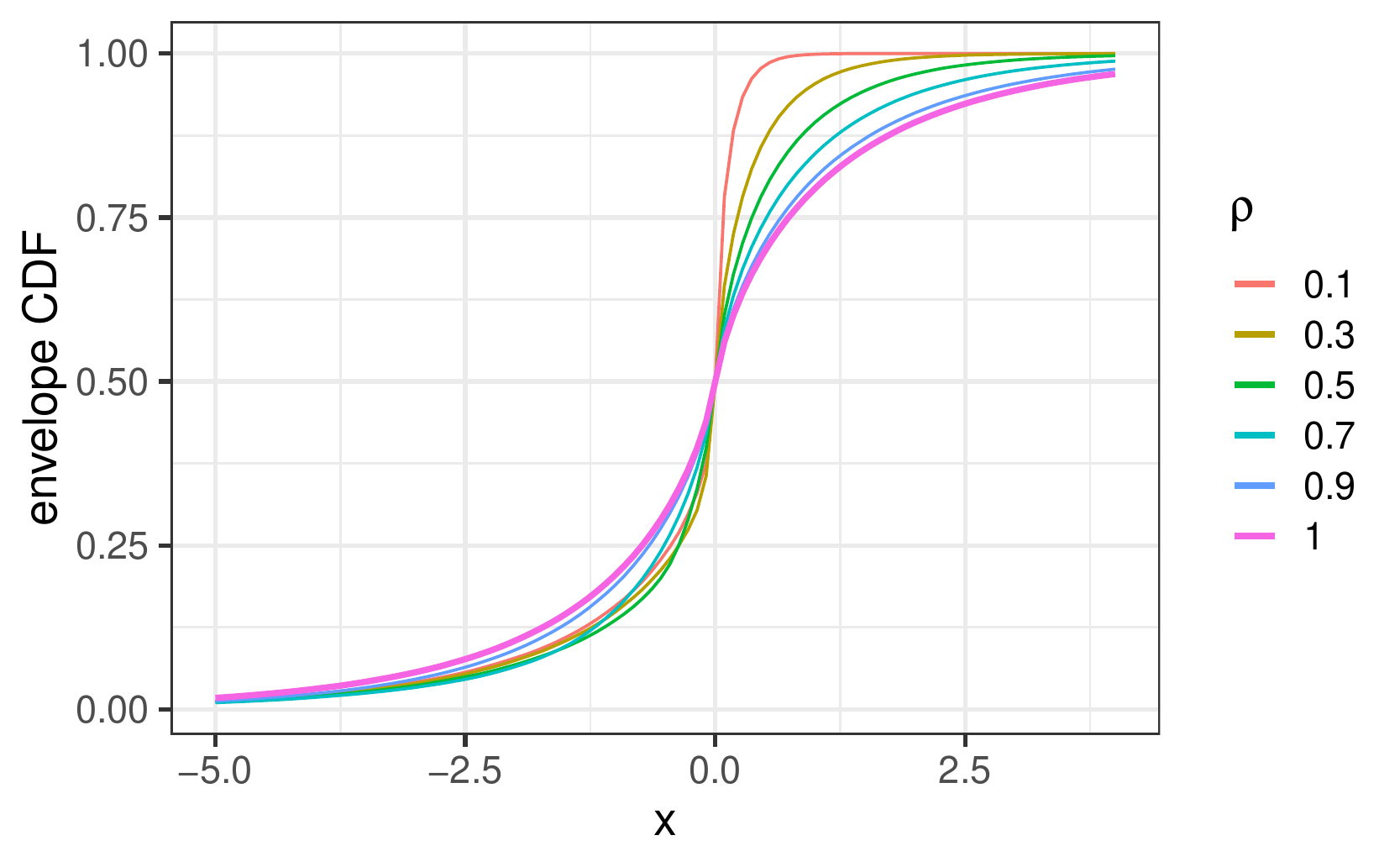}
\caption{The negative part of envelope of $\{\bar{F}_{\rho}: \rho \in [0,1]\}$ is the negative part of $\bar{F}_{\rho=1}$.}
\label{fig:envelope-of-envelopes}
\end{figure}

 \section{Model selection procedures} \label{sec:procedure}
Since we are selecting between two non-nested models, we want to refrain from choosing one of them as the default (the null hypothesis). By treating $\model_0$ and $\model_1$ symmetrically, however, a procedure that takes output value in $\{\model_0, \model_1\}$ cannot simultaneously control both types of error under a given tolerance. It can be seen from \cref{fig:asymp-dist-ws,fig:asymp-WW} that there are cases where the asymptotic distributions of $\lambda_n^{(0:1)}$ under $P_{\Sigma_{n}^{(0)}}$ and $P_{\Sigma_{n}^{(1)}}$ significantly overlap. In such cases, insisting on a dichotomous choice will inevitably result in a high probability of error under at least one model.   

To deal with this possible indistinguishability, we opt for a procedure with \emph{three options}: if two models can be sufficiently distinguished, it selects one of them; otherwise it refrains from commitment by selecting \emph{both models}, formally denoted as the union $\model_0 \cup \model_1$. It is worth stressing that we always assume at least one of the two models is true. By such a design, when the procedure does not output the union, we are ensured that the probability of choosing the wrong model is small, being controlled below a given tolerance $\alpha$. In contrast, in the usual hypothesis testing framework where supposedly $\model_0$ is the null and $\model_1$ is the alternative, one typically cannot simultaneously control both type-I and type-II error. In other words, our procedure selects model with ``confidence''. Recently the same notion has been investigated by \citet{lei2014classification} in a classification setting; \citet{robins2003uniform} also allows a test to make no decision when faced with ambiguity. We formalize the concept as follows. 

Suppose $(X_1, X_2, X_3) \in \mathbb{R}^{n \times 3}$ consists of $n$ independent samples from $\N(\bm{0}, \Sigma_n)$, where $\Sigma_n \in \model_0 \cup \model_1$ is allowed to change with $n$ and $\Sigma_n \rightarrow \Sigma^{\ast}$. The sequence $\Sigma_n$ models signal strength relative to the sample size. We consider a deterministic decision rule
\begin{equation}
\phi_n(S_n): \PSD^{3 \times 3} \rightarrow \{\model_0, \model_1, \model_0 \cup \model_1\},
\end{equation}
where sample covariance $S_n$ is the sufficient statistic. For a given sequence of $\Sigma_n \in \model_i \setminus \model_{1-i}$ with $\Sigma_n \rightarrow \Sigma^{\ast} \in \model_{i} \cup \model_{1-i}$, we define the asymptotic (type-I) error of $\phi_n$ as the large-sample probability of rejecting the true model, i.e.,
\begin{equation}
p_{\err}((\Sigma_n)):=\limsup_{n \rightarrow \infty} \Pr(\phi_n(S_n) = \model_{1-i}),
\end{equation}
where the probability is taken under $P_{\Sigma_n}^{n}$. Similarly, the asymptotic power is defined as
\begin{equation}
p_{\pow}((\Sigma_n)) := \liminf_{n \rightarrow \infty} \Pr(\phi_n(S_n) = \model_{i}).
\end{equation}
We say that the error is \emph{uniformly} controlled below a given size $\alpha > 0$, if
\begin{equation}
p_{\err}^{(0)}:=\sup_{(\Sigma_n^{(0)})} p_{\err}((\Sigma_n^{(0)})) \leq \alpha \quad \text{and} \quad p_{\err}^{(1)}:=\sup_{(\Sigma_n^{(1)})} p_{\err}((\Sigma_n^{(1)})) \leq \alpha,
\end{equation}
where for $i=0,1$ the supremum for $(\Sigma_n^{(i)})$ is taken over all converging sequences of $\Sigma_n^{(i)}$ within $\model_{i} \setminus \model_{1-i}$ (the limit can be in either $\model_i \setminus \model_{1-i}$ or $\model_{i} \cap \model_{1-i}$). In general, the power $p_
{\pow}((\Sigma_n))$ depends on the sequence considered and we do not seek power optimality or guarantee in a uniform sense. In the next section, we will compare the power of several proposed procedures to the theoretical optimal for $\Sigma_n$ considered in the two regimes of local asymptotics. 

By construction, using the $\alpha$-quantile of the envelope as the decision boundary achieves uniform error control. Based on the envelope of envelopes, a simple \emph{uniform rule} is
\begin{equation}
\phi_n^{\text{unif}} = \begin{cases}
\model_0, &\quad \lambda_n^{(0:1)} > -\bar{F}_{\rho=1}^{-1}(\alpha) \\
\model_1, &\quad \lambda_n^{(0:1)} < \bar{F}_{\rho=1}^{-1}(\alpha) \\
\model_0 \cup \model_1, &\quad \text{otherwise}
\end{cases}.
\end{equation}
To gain more power, since $\bar{F}_{\rho}$ is continuous in $\rho$ and $\rho$ can be consistently estimated (recall that $\rho = \rho_{\text{strong}}$ in the weak-strong regime, and $\rho=0$ in the weak-weak regime), an \emph{adaptive rule} can be formed as 
\begin{equation}
\phi^{\text{ada}}_n = \begin{cases}
\model_0, &\quad \lambda_n^{(0:1)} > -\bar{F}_{\hat{\rho}_n}^{-1}(\alpha)\\
\model_1, &\quad \lambda_n^{(0:1)} < \bar{F}_{\hat{\rho}_n}^{-1}(\alpha) \\
\model_0 \cup \model_1, &\quad \text{otherwise}
\end{cases},
\end{equation}
where $\hat{\rho}_n = |\hat{\rho}_{13,n}| \vee |\hat{\rho}_{23,n}|$ is the MLE for $|\rho|$. If it is desired to report a $p$-value, consider a potentially conservative $p\text{-value} = \bar{F}_{\rho}(-|\lambda_n^{(0:1)}|)$. For $\rho = 1$ and $\rho = \hat{\rho}_n$ respectively, the uniform rule and the adaptive rule can be then restated as 
\begin{equation*}
\phi_{n} = \begin{cases}
\model_0, &\quad \lambda_n^{(0:1)} > 0 \text{ and } p\text{-value} < \alpha \\
\model_1, &\quad \lambda_n^{(0:1)} < 0 \text{ and } p\text{-value} < \alpha \\
\model_0 \cup \model_1, &\quad \text{otherwise}
\end{cases}.
\end{equation*}
The conservative $p$-value can be computed numerically by Monte Carlo and then taking the maximum over a grid of values for $\gamma$.

\begin{theorem} \label{thm:error-control}
The adaptive rule $\phi^{\text{ada}}_n$ controls asymptotic error uniformly below $\alpha$ for $0 < \alpha < 1/2$. 
\end{theorem}
\begin{proof}
We show error guarantee when $\model_0$ is true. The same argument holds when $\model_1$ is true. It suffices to show for any converging sequence $\Sigma_n \in \model_{0} \setminus \model_{1}$, 
\begin{equation*}
p_{\err}((\Sigma_n)) = \limsup_{n \rightarrow \infty} \Pr(\phi_n(S_n) = \model_{1}) \leq \alpha,
\end{equation*}
where the probability is measured under $P_{\Sigma_n}^n$. Suppose $\Sigma_{n} \rightarrow \Sigma^{\ast} \in \model_0 \cup \model_1$. If $\Sigma^{\ast} \notin \model_0 \cap \model_1$, then $\lambda_n^{(0:1)}$ is unbounded in probability towards $+\infty$. Hence $\Pr(\phi_n(S_n) = \model_0) \rightarrow 1$ and $p_{\err}(\Sigma_n) = 0$. In the following we prove the claim for $\Sigma^{\ast} \in \model_0 \cap \model_1$. Suppose $\hat{\rho}_{ij,n}$, $\rho_{ij,n}$ and $\rho_{ij}$ respectively denote the corresponding correlation coefficient of $S_n$, $\Sigma_n$ and $\Sigma^{\ast}$. We have three cases depending on the rate at which $\Sigma_n$ converges. 

\begin{enumerate}
\item When $|\rho_{13,n} \rho_{23,n}| \asymp 1/\sqrt{n}$, there are two regimes depending on $\Sigma^{\ast}$. 
\begin{enumerate}
\item In the weak-strong regime, without loss of generality suppose $\sqrt{n} \rho_{13,n} \rightarrow \gamma \neq 0$ and $\rho_{23,n} \rightarrow \rho \neq 0$. By consistency $\hat{\rho}_n = |\hat{\rho}_{13,n}| \vee |\hat{\rho}_{23,n}| \rightarrow_{p} |\rho|$ and the definition of envelope, we have
\begin{equation*}
\limsup_{n} \Pr(\lambda_n^{(0:1)} < \bar{F}_{\hat{\rho}_n}^{-1}(\alpha)) = F_{\gamma,\rho}(\bar{F}_{\rho}^{-1}(\alpha)) \leq F_{\gamma, \rho}(F_{\gamma, \rho}^{-1}(\alpha)) = \alpha.
\end{equation*}

\item In the weak-weak regime, suppose $\sqrt{n} \rho_{13,n} \rho_{23,n} \rightarrow \delta \neq 0$. We have $\hat{\rho}_n = |\hat{\rho}_{13,n}| \vee |\hat{\rho}_{23,n}| = (|\rho_{13,n}| \vee |\rho_{23,n}|) + O_p(1/\sqrt{n}) \rightarrow_{p} 0$ since both $\rho_{13,n}, \rho_{23,n} \rightarrow 0$. By \cref{prop:envelope-continuity}, we have
\begin{equation*}
\begin{split}
\limsup_{n} \Pr(\lambda_n^{(0:1)} < \bar{F}_{\hat{\rho}_n}^{-1}(\alpha)) &= G_{\delta}(\bar{F}_{\rho=0}^{-1}(\alpha))\\
&= G_{\delta}(\bar{G}^{-1}(\alpha)) \leq G_{\delta}(G_{\delta}^{-1}(\alpha)) = \alpha.
\end{split}
\end{equation*}
\end{enumerate}
\item When $|\rho_{13,n} \rho_{23,n}| = o(1/\sqrt{n})$, we have $\lambda_{n}^{(0:1)} \rightarrow_{p} 0$. Since $F^{-1}_{\rho}(\alpha) < c < 0$ for $\alpha < 1/2$, we have $\Pr(\phi_{n}(S_n) = \model_0 \cup \model_1) \rightarrow 1$. 

\item When $|\rho_{13,n} \rho_{23,n}| = \omega(1/\sqrt{n})$, we have $\Pr(\lambda_{n}^{(0:1)} > c) \rightarrow 1$ for any constant $c$ and hence $\Pr(\phi_n(S_n) = \model_0) \rightarrow 1$.
\end{enumerate}
\end{proof}

As can be seen from the proof, the consistency of model selection based on the loglikelihood (or AIC/BIC since in this case $\model_0$ and $\model_1$ have the same dimensions) is a special case when $|\rho_{13,n} \rho_{23,n}| = \omega(n^{-1/2})$, i.e., under strong signal or large enough sample size. However, under $|\rho_{13,n} \rho_{23,n}| = O(n^{-1/2})$, simply choosing the model with the highest loglikelihood (or the lowest AIC/BIC) can lead to large errors, as we will illustrate in the next section. Note that \cref{thm:error-control} provides a ``rate-free'' guarantee, in the sense that it does \emph{not} require any \textit{a priori} assumption on the rate of signal strength relative to the sample size. The envelope of envelopes leads to the same guarantee for the uniform rule. 

\begin{corollary}
The decision rule $\phi_n^{\text{unif}}$ controls asymptotic error uniformly below $\alpha$ for $0 < \alpha < 1/2$.
\end{corollary}
\begin{proof}
It follows from \cref{thm:error-control} and \cref{prop:envelope-of-envelopes}.
\end{proof}

The uniform rule can be easily applied by comparing the difference in log-likelihoods to a single number, e.g., 3.19 for $\alpha=0.05$ and 5.97 for $\alpha = 0.01$. The adaptive rule can be implemented by numerically evaluating $\bar{F}_{\rho}^{-1}(\alpha)$ via Monte Carlo on a grid of $\rho$ and interpolating. Some values are plotted in \cref{fig:envelope-quantiles} and tabulated in \cref{tab:envelope-quantiles} based on $10^7$ samples. It is interesting to notice that $\bar{F}^{-1}_{\rho}(\alpha)$ is not monotonic in $\rho \in [0,1]$. 

\begin{figure}[!htb]
\includegraphics[width=0.8\textwidth]{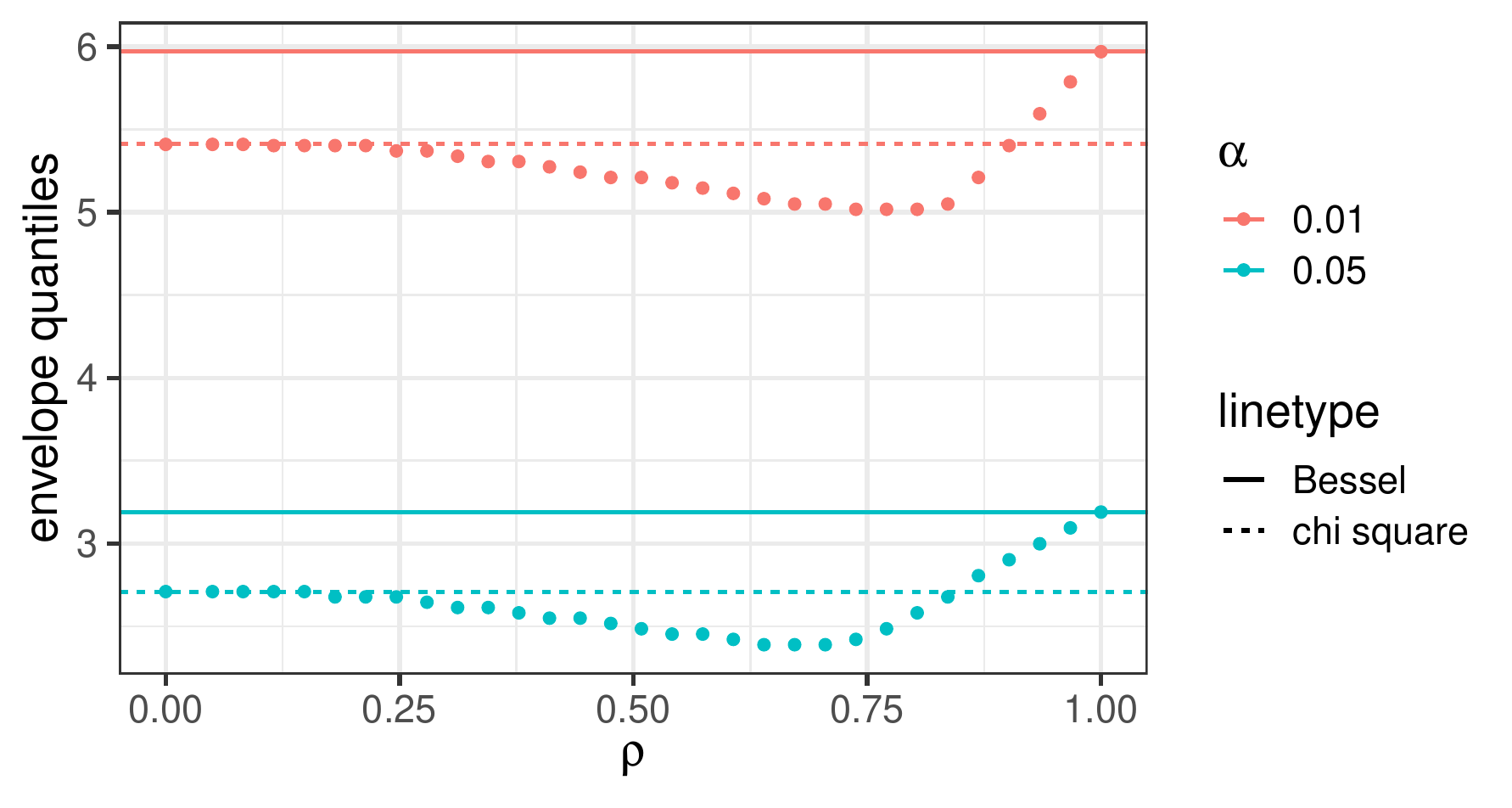}
\caption{Negated $\alpha$-quantiles of $\bar{F}_{\rho}$ evaluated on a grid.}
\label{fig:envelope-quantiles}
\end{figure}

\begin{table}[!htb]
\caption{Envelope quantiles $-\bar{F}_{\rho}^{-1}(\alpha)$, Monte Carlo standard errors $\leq 0.01$}
\label{tab:envelope-quantiles}
\begin{tabular}{lccccccccccc}
\toprule
$\rho$ & 0.0 & 0.1 & 0.2 & 0.3 & 0.4 & 0.5 & 0.6 & 0.7 & 0.8 & 0.9 & 1.0\\
\midrule
$\alpha=0.05$ & 2.71 & 2.71 & 2.68 & 2.65 & 2.58 & 2.48 & 2.42 & 2.39 & 2.58 & 2.90 & 3.19\\
$\alpha=0.01$ & 5.41 & 5.41 & 5.40 & 5.34 & 5.27 & 5.21 & 5.11 & 5.05 & 5.02 & 5.40 & 5.97\\
\bottomrule
\end{tabular}
\end{table}
 \section{Simulations} \label{sec:simu}
In this section we conduct numerical simulations to assess the performance of the adaptive and uniform decision rules proposed in the previous section. In subsequent simulations we use $\alpha=0.05$. In addition to the two methods we propose, we also consider the following methods for comparison. 

\paragraph{Naive} The naive procedure selects the model with a higher likelihood (or equivalently, a lower AIC/BIC, since $\model_0$ and $\model_1$ have the same dimensions), namely
\begin{equation*}
\phi_n^{\text{naive}} = \begin{cases}
\model_0, \quad \lambda_n^{(0:1)} > 0\\
\model_1, \quad \lambda_n^{(0:1)} < 0 
\end{cases}.
\end{equation*}
This is effectively choosing a single model based on AIC/BIC since the penalty terms cancel out as $\model_0$ and $\model_1$ have the same dimension.

\paragraph{Interval Selection} This method is adapted from \citet{drton2004model}. We construct $(1-\alpha)$-level non-simultaneous confidence intervals on correlation coefficients $\rho_{12}$ and $\rho_{12 \cdot 3}$ with Fisher's $z$-transform \citep{fisher1924distribution}. The decision rule is 
\begin{equation}
\phi_n^{\text{interval}} = \begin{cases}
\model_0, &\quad 0 \in \text{C.I.}(\rho_{12}) \text{ and } 0 \notin \text{C.I.} (\rho_{12 \cdot 3})  \\
\model_1, &\quad 0 \in \text{C.I.} (\rho_{12 \cdot 3}) \text{ and } 0 \notin \text{C.I.} (\rho_{12}) \\
\model_0 \cup \model_1, &\quad \text{otherwise}
\end{cases}.
\end{equation}
Note that the interval selection method controls asymptotic error below $\alpha$. For example, when $\model_0$ is true, 
\begin{equation*}
\limsup_{n} \Pr(\phi_n^{\text{interval}} = \model_1) \leq \limsup_{n} \Pr(0 \notin \text{C.I.} (\rho_{12})) \leq \alpha.
\end{equation*}

We conduct numerical simulations in the following three settings.

\subsection{Local hypotheses} We simulate under $\Sigma_n^{(0)}$ and $\Sigma_n^{(1)}$ (variances are set to unity) for the two regimes considered in \cref{sec:local-asymp}. The power is compared to the theoretically optimal. Since exact values of the total variation distance are intractable, we plot bounds given by \cref{eqs:bound-power} in grey curves. We perform 4,000 replications for each point on the graphs.

See \Cref{fig:size-local-ws,fig:power-local-ws} for the size and power in the weak-strong regime (\cref{eqs:sigma-0-ws,eqs:sigma-1-ws}) under $n = 1,000$. Smaller sample sizes $n=100,200,\dots$ generate very similar results. See \Cref{fig:size-power-local-ww} for the size and power in the weak-weak regime (\cref{eqs:sigma-0-ww,eqs:sigma-1-ww}), where we set $\rho_{13,n}=n^{-a/4}$, $\rho_{23,n}=n^{-1/2+a/4}$ and let $a$ vary. We observe that (i) the naive method does not control error at all; (ii) the other three methods control error uniformly even under relatively small $n$. We also observe that the relation ``adaptive'' $>$ ``uniform'' $>$ ``interval'' holds in general in terms of both size and power. By comparing to the grey curves, we regard the adaptive rule as achieving near-optimal power in these settings. 

\begin{figure}[htbp]
\includegraphics[width=.9\textwidth]{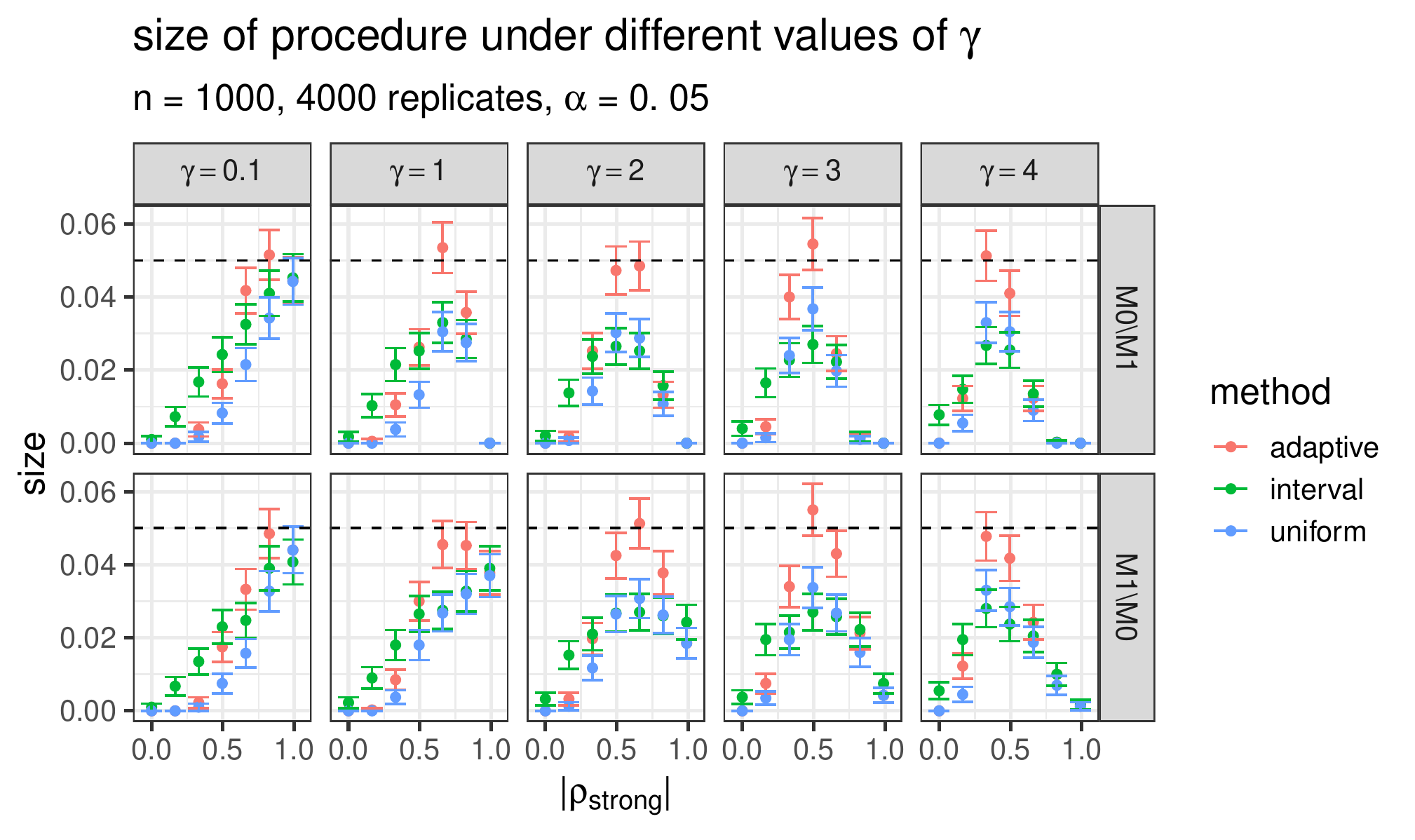}
\includegraphics[width=.9\textwidth]{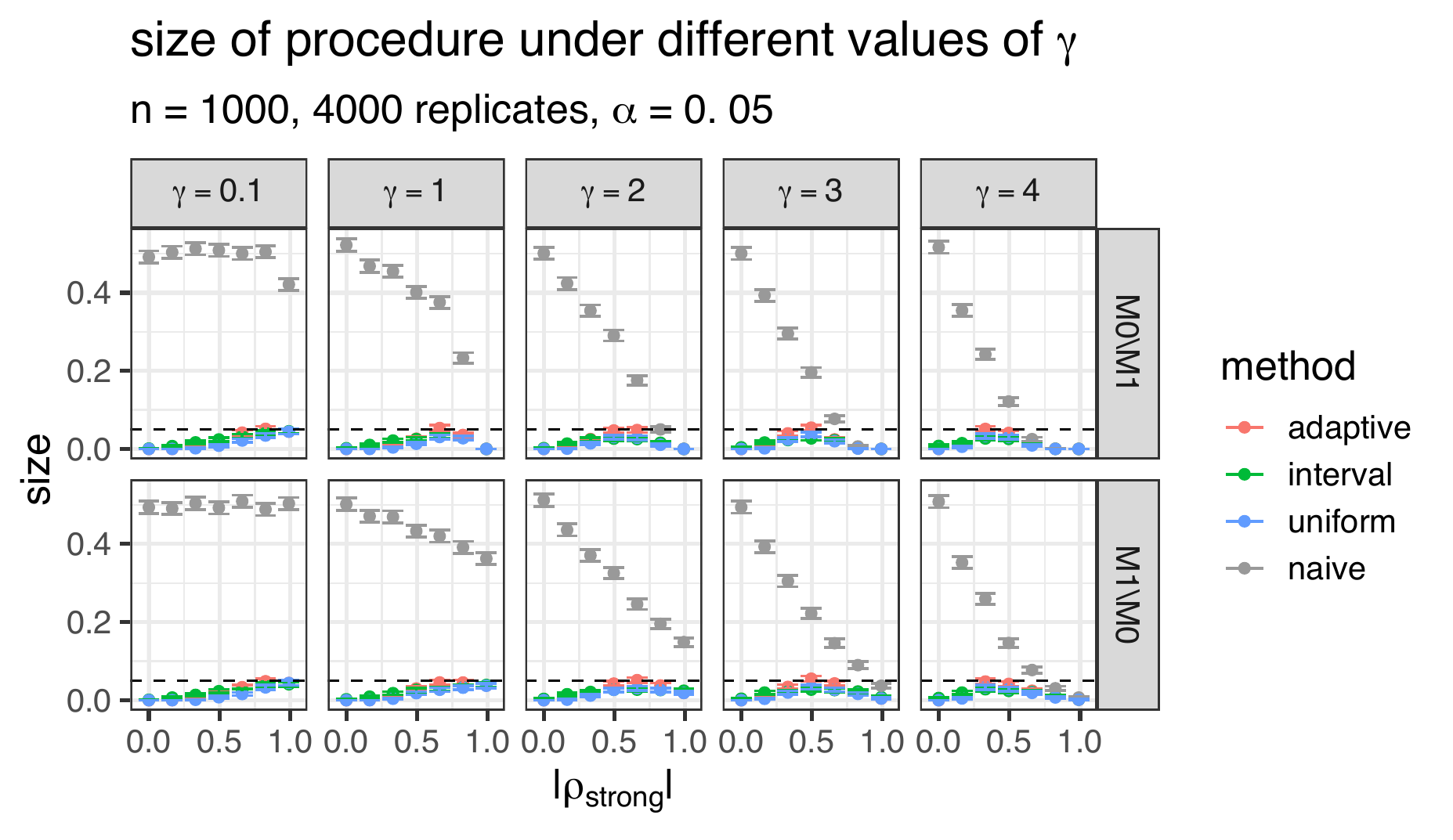}
\caption{Size $\Pr(\phi_n = \mathcal{M}_{1-i} | \mathcal{M}_i)$ of the procedures (with 95\% confidence intervals) under the weak-strong regime of local hypotheses. $\alpha = 0.05$ is marked as dashed. The naive method is only included in the second plot for better visualization.}
\label{fig:size-local-ws}
\end{figure}

\begin{figure}[htbp]
\includegraphics[width=.9\textwidth]{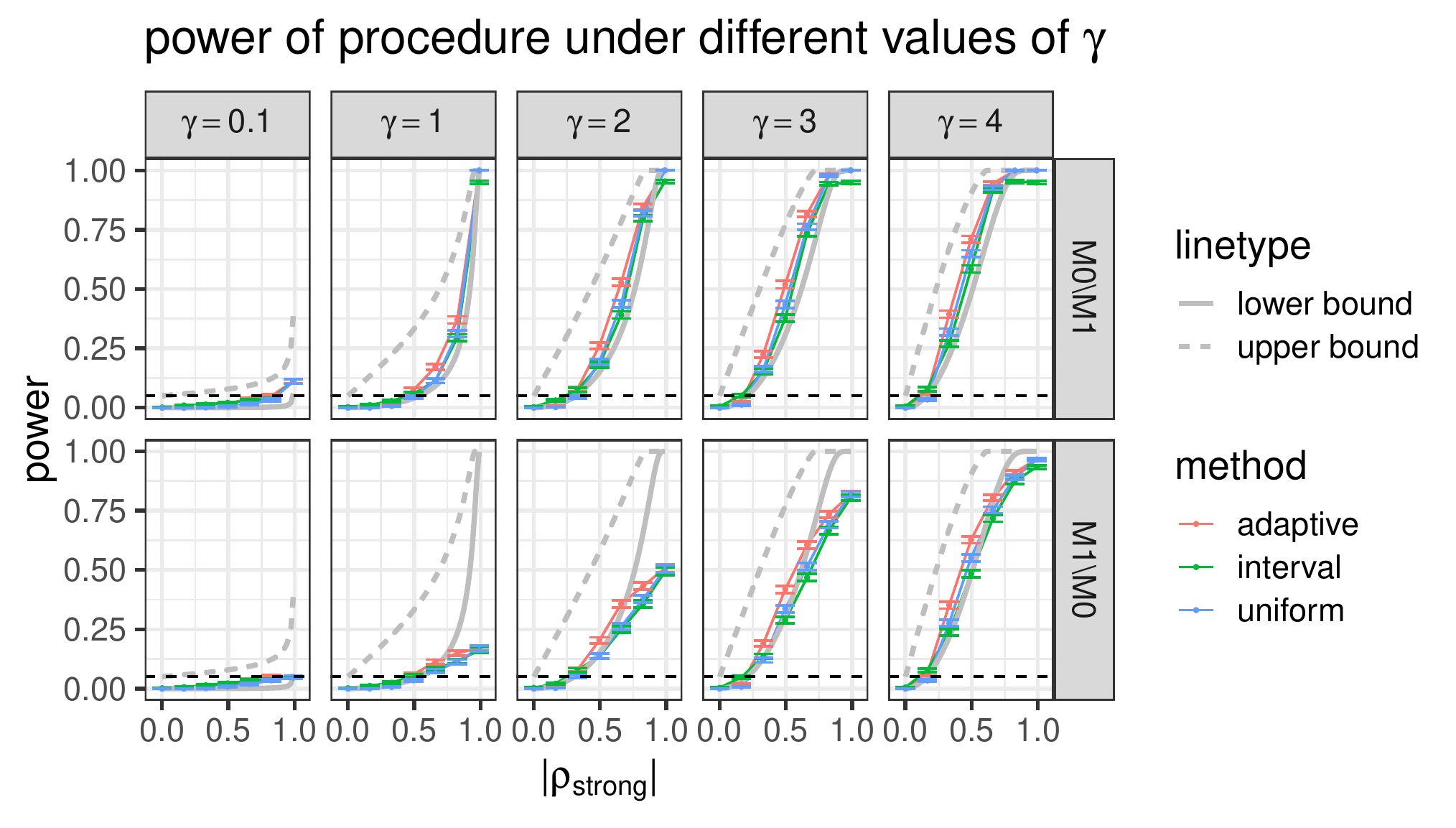}
\caption{Power $\Pr(\phi_n = \mathcal{M}_{i} | \mathcal{M}_i)$ of the procedures (with 95\% confidence intervals) under the weak-strong regime of local hypotheses. $\alpha = 0.05$ is marked as dashed. Grey curves are bounds on the theoretically optimal power.}
\label{fig:power-local-ws}
\end{figure}

\begin{figure}[htbp]
\includegraphics[width=.9\textwidth]{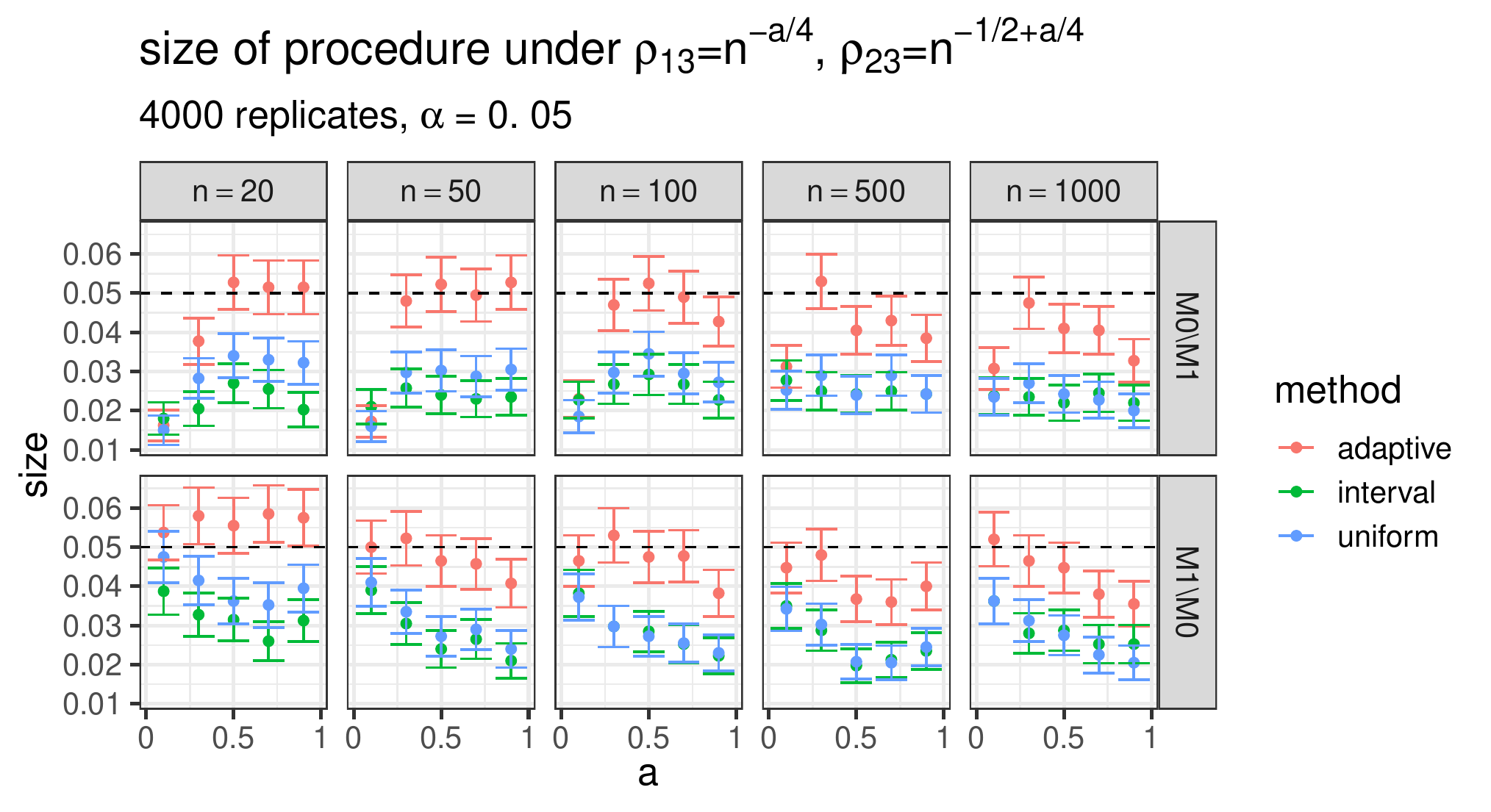}
\includegraphics[width=.9\textwidth]{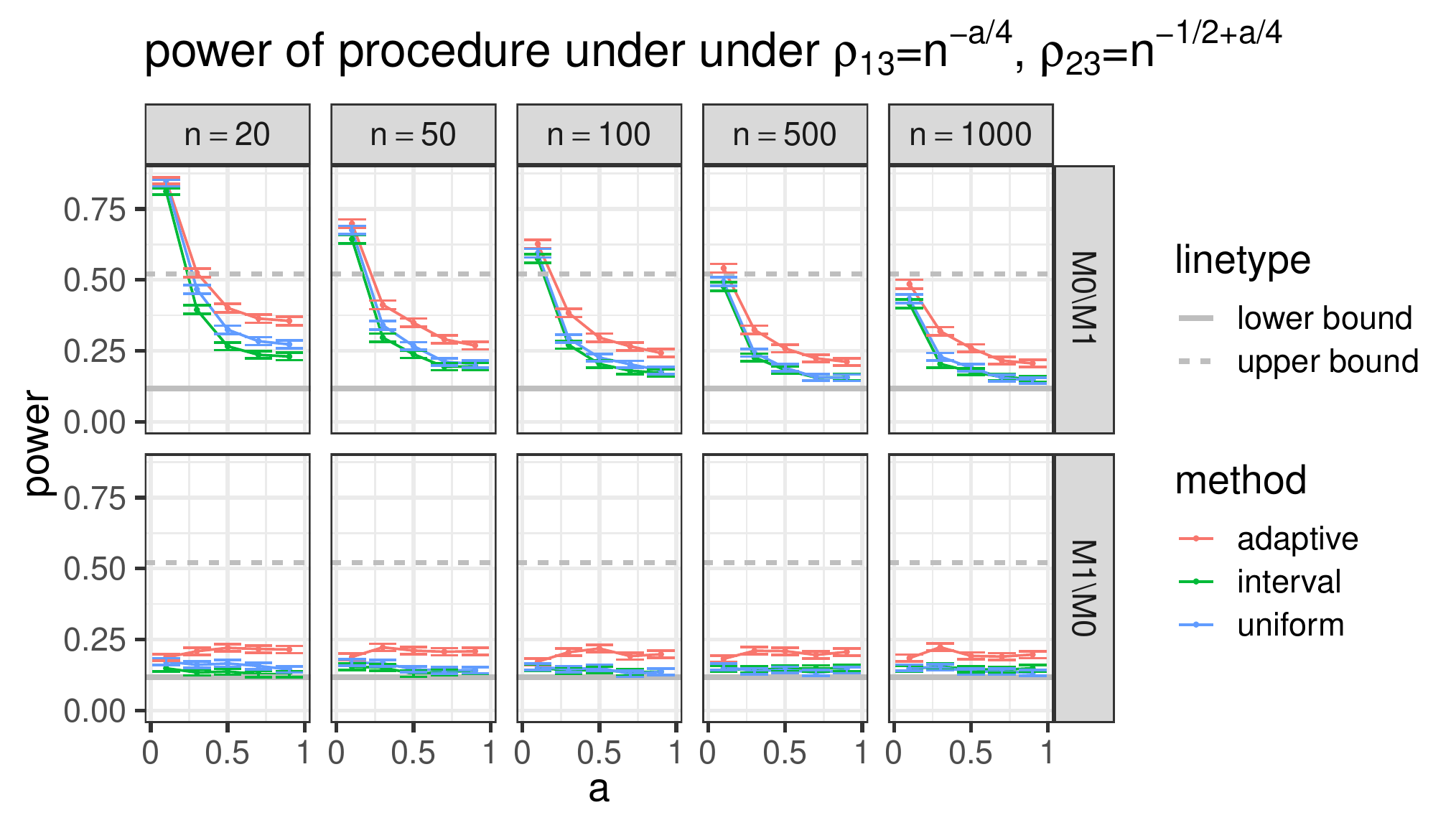}
\caption{Size $\Pr(\phi_n = \mathcal{M}_{1-i} | \mathcal{M}_i)$ and power $\Pr(\phi_n = \mathcal{M}_{i} | \mathcal{M}_i)$ of the procedures (with 95\% confidence intervals) under the weak-weak regime of local hypotheses ($\rho_{13,n} \rho_{23,n} = n^{-1/2}$). $\alpha = 0.05$ is marked as dashed. Grey lines are bounds on the theoretically optimal power in the second plot. The naive method is excluded due to its large type-I error.}
\label{fig:size-power-local-ww}
\end{figure}

\subsection{Projected Wishart} We generate a covariance matrix by firstly drawing $\tilde{\Sigma}$ from the Wishart distribution (with the scale matrix chosen as $\sigma_{ij} = (-1/2)^{|i-j|}$) and then projecting $\tilde{\Sigma}$ into $\model_0$ or $\model_1$ respectively by finding the MLE under each model. Then we perform model selection based on two sets of zero-mean Gaussian samples generated with the two projected covariances respectively. We vary the degrees of freedom for the Wishart distribution. See \Cref{fig:wishart} for the results. 

\begin{figure}[htbp]
\includegraphics[width=.9\textwidth]{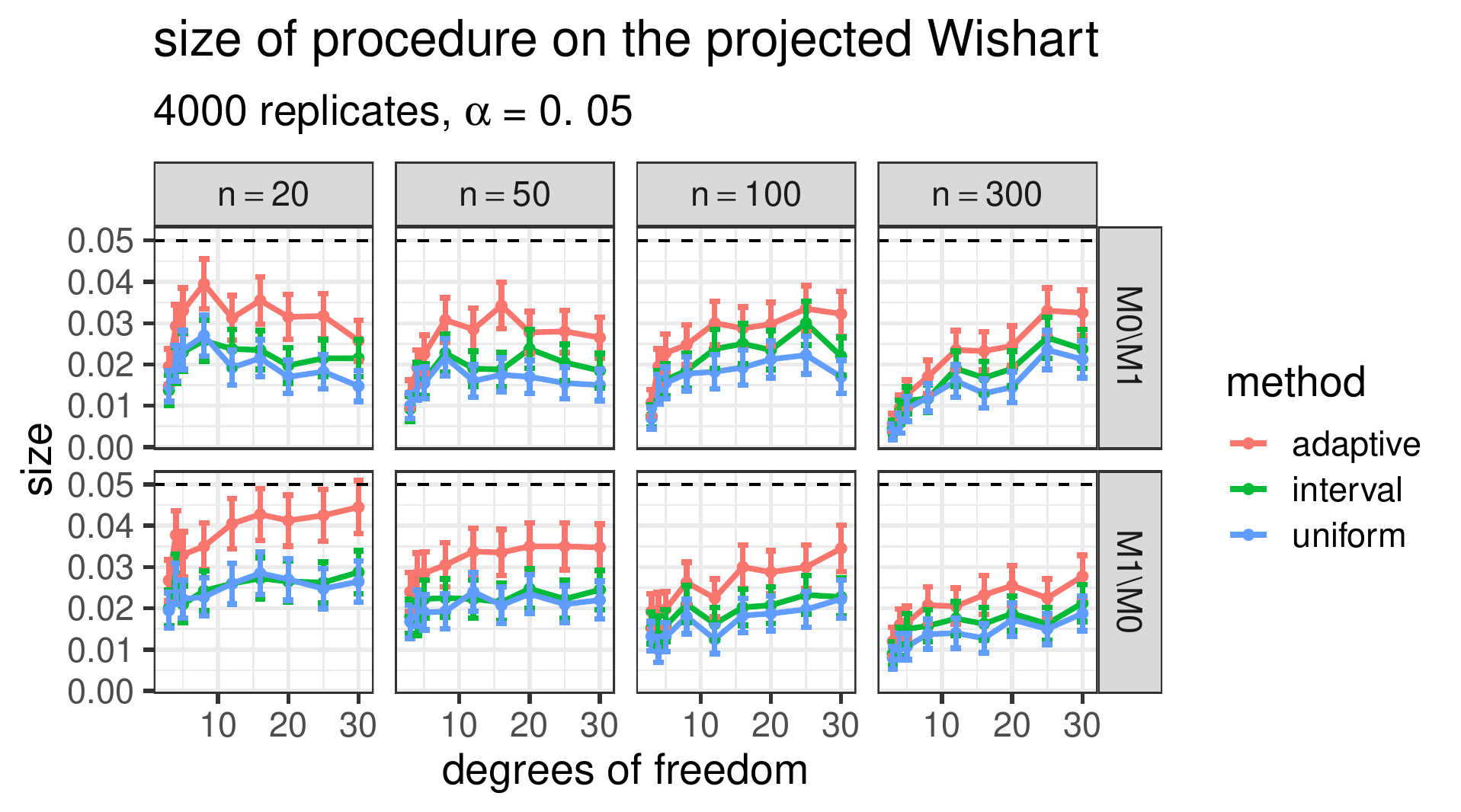}
\includegraphics[width=.9\textwidth]{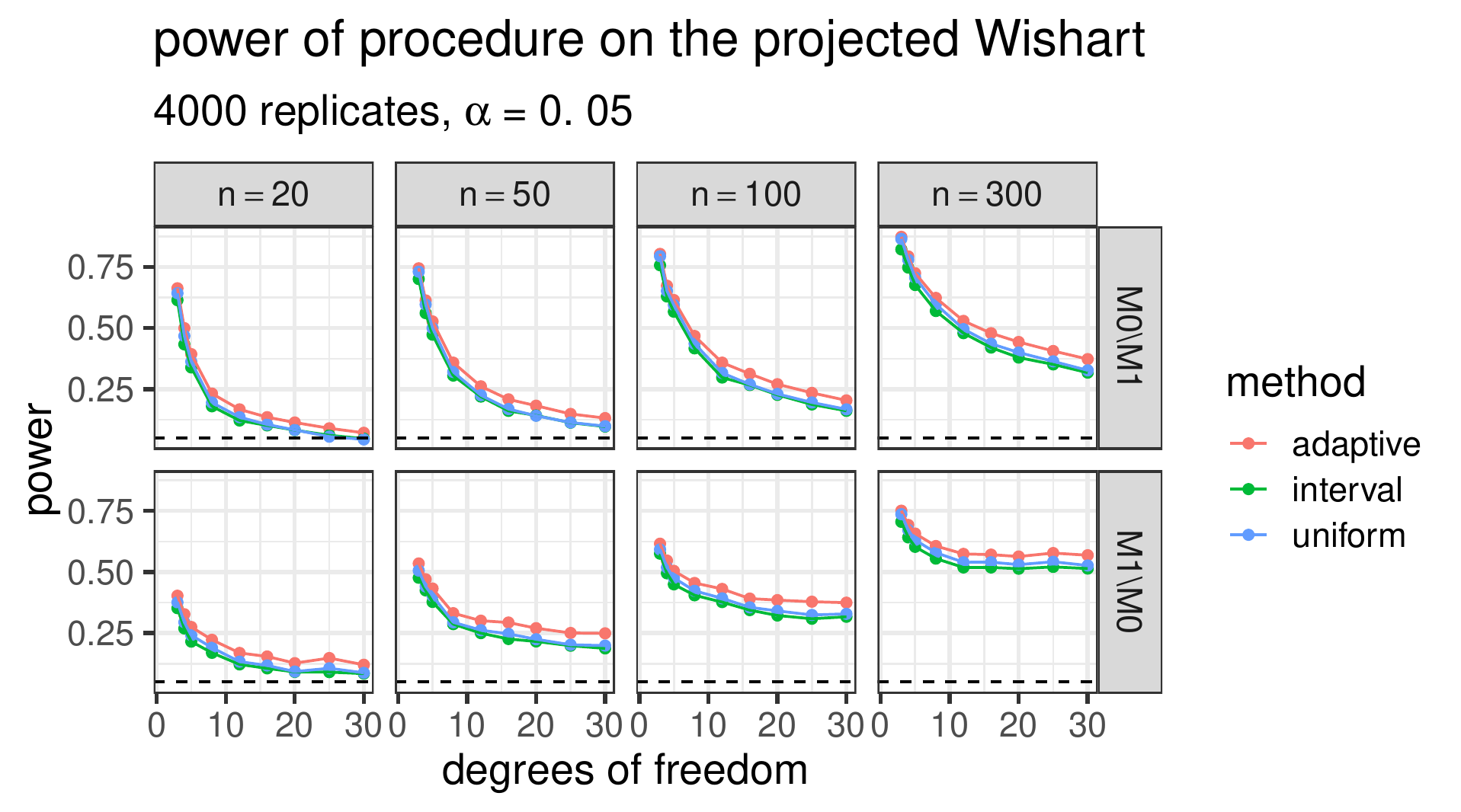}
\caption{Size $\Pr(\phi = \mathcal{M}_{1-i} | \mathcal{M}_i)$ and power $\Pr(\phi = \mathcal{M}_i | \mathcal{M}_i)$ of the procedures on projected Wishart matrices (with 95\% confidence intervals). $\alpha = 0.05$ is marked as dashed. The naive method makes large errors and is excluded.}
\label{fig:wishart}
\end{figure}

\subsection{Conditional on covariates} We consider the common regression setting where two types of independences are contrasted conditional on a set of covariates $X \in \mathbb{R}^p$. In other words, we want to select between $\mathcal{M}_0: Y_1 \indep Y_2 \mid X$ and $\mathcal{M}_1: Y_1 \indep Y_2 \mid Y_3, X$. We generate instances by 
\begin{equation}
(Y_1, Y_2, Y_3) = X^{\T} (\beta_1, \beta_2, \beta_3) + E, \quad E \sim \N(\bm{0}, \Sigma),
\end{equation}
where we use the previous projected Wishart to generate error covariance $\Sigma$ under $\model_0$ and $\model_1$. We perform model selection by firstly regressing $(Y_1, Y_2, Y_3)$ onto $X$ with least squares and then apply the model selection procedures to the residual covariance. Covariates are randomly drawn from standard Gaussians and regression coefficients are generated from a $t$-distribution with $4$ degrees of freedom. We fix $n=1,000$ and vary the number of covariates $p$. The results are presented in \Cref{fig:regression}. We observe that the proposed procedure continues to maintain nominal size until $p$ is relatively large compared to $n$. The power performance, on the other hand, does not seem to vary much as $p$ grows. 

\begin{figure}[htbp]
\includegraphics[width=.8\textwidth]{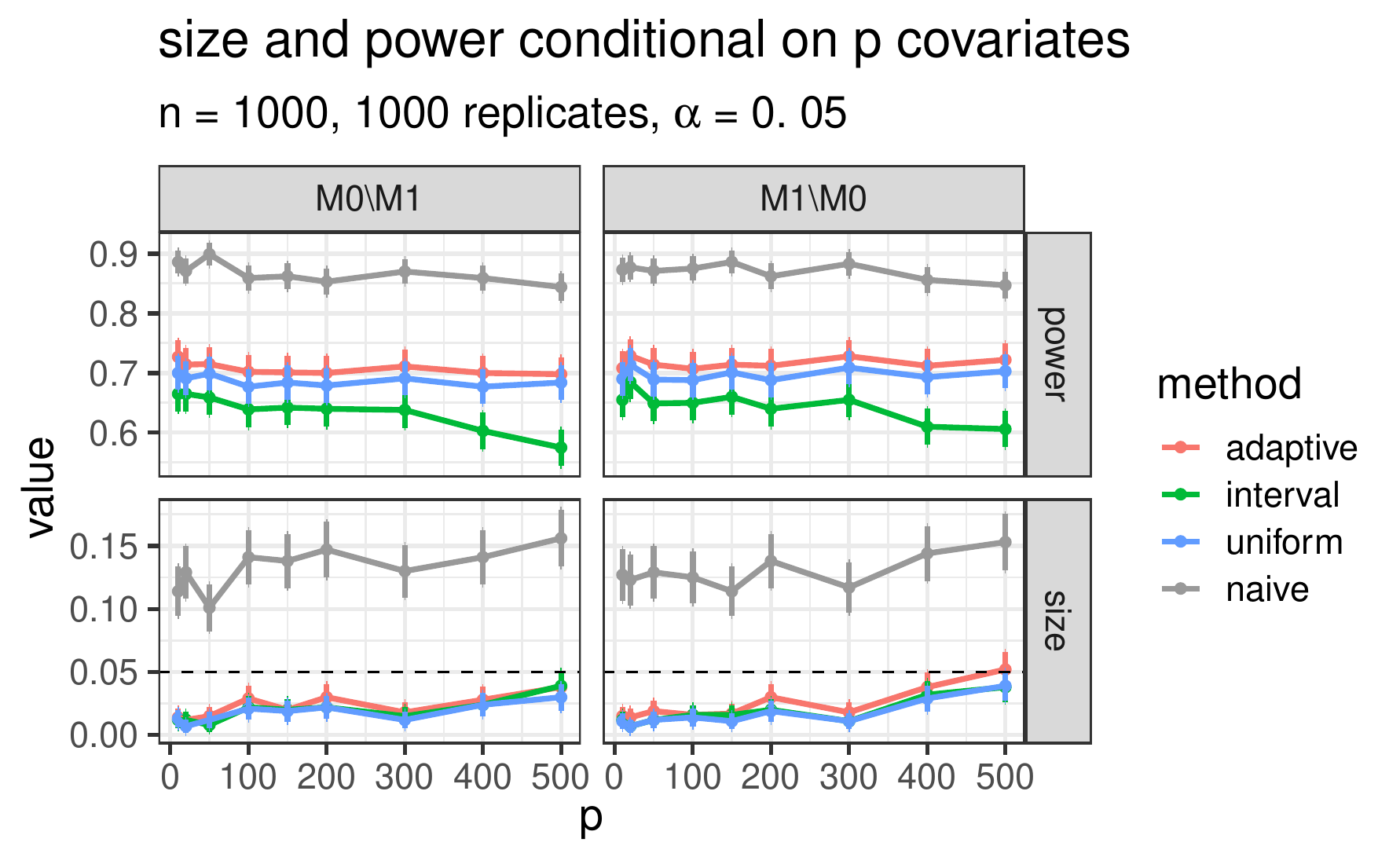}
\caption{Size $\Pr(\phi = \mathcal{M}_{1-i} | \mathcal{M}_i)$ and power $\Pr(\phi = \mathcal{M}_i | \mathcal{M}_i)$ of the model selection procedures conditioned on $p$ covariates (with 95\% confidence intervals). Error covariances are generated from the projected Wishart. The procedures are applied to the least-squares residual covariance. $\alpha = 0.05$ is marked as dashed.}
\label{fig:regression}
\end{figure}
 \section{Real data example} \label{sec:real-data}
In this section we showcase an example of applying the method to edge orientation in learning a DAG. In studying the American occupational structure, \citet{blau1967american} measured the following covariates on $n = 20,700$ subjects:
\begin{itemize}
\item[] $V$: father's educational attainment,
\item[] $X$: father's occupational status,
\item[] $U$: child's educational attainment,
\item[] $W$: status of child's first job,
\item[] $Y$: status of child's occupation in 1962.
\end{itemize}
The data is summarized as the following correlation matrix of $(V,X,U,W,Y)$
\begin{equation*}
S_n = \begin{pmatrix}
1.000 & 0.516 & 0.453 & 0.332 & 0.322\\
0.516 & 1.000 & 0.438 & 0.417 & 0.405\\
0.453 & 0.438 & 1.000 & 0.538 & 0.596\\
0.332 & 0.417 & 0.538 & 1.000 & 0.541\\
0.322 & 0.405 & 0.596 & 0.541 & 1.000\\
\end{pmatrix}.
\end{equation*}
At level $\alpha = 0.01$, the PC algorithm identifies the skeleton by $d$-separation, which only removes the edge between $V$ and $Y$ based on $Y \indep V \mid U, X$. This is because the PC algorithm tests for conditional independence given smaller conditioning sets first. By a common-sense temporal ordering $\{V,X\} < U < \{W, Y\}$ among the variables, edges can be oriented except for $X-V$ and $W-Y$; see \cref{fig:blau}. The edge $V-X$ does not involve a collider and the orientation is statistically unidentifiable. 

\begin{figure}[!htb]
\begin{tikzpicture}[rv/.style={ellipse, draw, very thick, minimum size=7mm}, node distance=25mm, >=stealth]
\node[rv] (1) {$V$};
\node[rv, right of=1] (2) {$U$};
\node[rv, below of=1] (3) {$X$};
\node[rv, right of=3] (4) {$W$};
\node[rv, right of=2, yshift=-12mm] (5) {$Y$};
\draw[->, very thick, color=blue] (1) -- (2);
\draw[->, very thick, color=blue] (1) -- (4);
\draw[very thick, color=red] (1) -- (3);
\draw[->, very thick, color=blue] (2) -- (5);
\draw[->, very thick, color=blue] (2) -- (4);
\draw[->, very thick, color=blue] (3) -- (2);
\draw[->, very thick, color=blue] (3) -- (4);
\draw[->, very thick, color=blue] (3) -- (5);
\draw[very thick, color=red] (4) -- (5);
\end{tikzpicture}
\caption{CPDAG inferred from \citet{blau1967american} dataset. The skeleton is inferred based on $d$-separation at level $\alpha = 0.01$ with the PC algorithm. Blue edges are oriented based on temporal ordering $\{V,X\} < U < \{W, Y\}$.}
\label{fig:blau}
\end{figure}
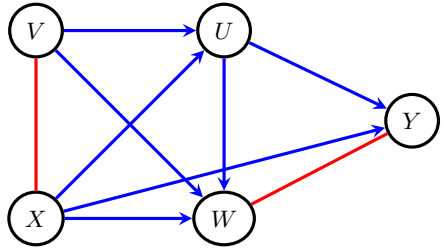

However, the orientation of $W-Y$ raises the interesting question of testing
\begin{equation*}
\model_0 \ (Y \rightarrow W): V \indep Y \mid U, X \quad \text{versus} \quad \model_1 \ (Y \leftarrow W): V \indep Y \mid W, U, X.
\end{equation*}
We apply our method to the conditional correlation of $(V,W,Y)$ given $(U,X)$. We have $\lambda_{n}^{(0:1)} = 3.72$ and $p\text{-value}=0.026$ under the envelope distribution $\bar{F}_{\hat{\rho}_n}$. Therefore, under $\alpha = 0.01$ the adaptive procedure would choose $\model_0 \cup \model_1$ and leave the orientation undetermined (the procedure would choose $Y \rightarrow W$ under $\alpha = 0.05$). This example illustrates the potential ambiguity in model selection even under a large sample size. The reader is referred to \citet[Section 5.8.4]{spirtes2000causation} for another discussion of the same example.

 \section{Discussion} \label{sec:discussion}
We have considered choosing between marginal independence and conditional independence in a Gaussian graphical model, assuming we know at least one of them is true. The loglikelihood ratio statistic converges to a tight law under a sequence of truths converging to the intersection of the two models at a certain rate. The asymptotic distribution is shown to be non-uniform as it depends on \emph{where} and \emph{how} the sequence converges. We address this non-uniformity issue by introducing a family of envelope distributions that are well-behaved and bring back the continuity of asymptotic laws, as indexed by a parameter that can be consistently estimated. Contrary to the usual Neyman--Pearson hypothesis testing, we treat the two models symmetrically and develop model selection rules that choose both models when they are indistinguishable under a given sample size. Such rules can be designed according to the quantiles of the envelope distributions to uniformly control the type-I error below a desired level. As noted before we believe that ``rate-free'' asymptotic guarantees that are uniform are more useful in practice, since they do not rely upon untestable assumptions regarding the sample size and the signal strength.

In this report we restricted ourselves to the Gaussian case. For testing conditional independence, some form of distributional assumption seems inevitable, since recent work of \citet{shah2018hardness} shows that testing conditional independence without restricting the form of conditional independence is impossible in general.

Selection of non-nested models routinely relies on penalized scores based on loglikelihoods, such as the negated AIC and BIC. However, as we show, in the context of a weak signal relative to the sample size, simply choosing the model with the highest score can lead to considerable errors. To select models with ``confidence'', one should also look at the ``gaps'' between the top scores. We believe that the method developed in this paper may be generalizable to a wider range of model selection problems.  
\section*{Acknowledgements}
RG thanks Michael Perlman for helpful discussions. TR thanks Robin Evans and Peter Spirtes. The research was supported by the U.S. Office of Naval Research.

\bibliographystyle{plainnat}

\end{document}